\documentclass[a4paper]{article}

\usepackage[english]{babel}
\usepackage[utf8x]{inputenc}
\usepackage[T1]{fontenc}
\usepackage{graphicx}
\usepackage{caption, subcaption}
\usepackage[title]{appendix}
\usepackage{tikz}
\usepackage{pgfplots}
\usepackage{cutwin}
\graphicspath{ {images/} }
\usetikzlibrary{intersections,decorations.pathreplacing,decorations.markings,calc,angles,quotes,arrows.meta,pgfplots.fillbetween,patterns}
\usepackage[a4paper,top=2.9cm,bottom=2cm,left=2.5cm,right=2.5cm,marginparwidth=1.75cm]{geometry}

\usepackage{amsmath,yhmath}
\usepackage{amsthm}
\usepackage{amssymb}
\usepackage{hyperref}
\usepackage{cleveref}
\usepackage{epstopdf}
\usepackage{comment}
\usepackage{todonotes}
\usepackage{color}
\usepackage{float}
\usepackage[sort&compress,numbers]{natbib}
\newtheorem{thm}{Theorem}
\newtheorem{lem}[thm]{Lemma}
\newtheorem{prop}[thm]{Proposition}

\theoremstyle{definition}
\newtheorem{defn}[thm]{Definition}
\newtheorem{rmk}[thm]{Remark}
\numberwithin{equation}{section}
\numberwithin{thm}{section}

\newcommand{\C}{\mathcal{C}}

\newcommand{\F}{\mathcal{F}}

\newcommand{\N}{\mathbb{N}}

\newcommand{\R}{\mathbb{R}}

\newcommand{\W}{\mathcal{W}}
\newcommand{\Z}{\mathbb{Z}}
\renewcommand{\r}{\mathrm{r}}
\newcommand{\p}{\partial}
\newcommand{\dx}{\: \mathrm{d}}

\newcommand{\ie}{\textit{i.e.}}
\newcommand{\nm}{\noalign{\smallskip}}
\newcommand{\ds}{\displaystyle}
\newcommand{\iu}{\mathrm{i}\mkern1mu}
\newcommand{\I}{\mathrm{I}}
\newcommand{\II}{\mathrm{II}}

\usepackage{soul}
\sethlcolor{cyan}
\usepackage{xcolor}
\newcommand{\mathhl}[1]{\colorbox{cyan}{$\ds #1$}}
\newcommand{\hlbox}[1]{\colorbox{cyan}{#1}}
\renewcommand\hl[1]{#1}
\renewcommand\mathhl[1]{#1}
\renewcommand\hlbox[1]{#1}

\makeatletter
\newcommand{\neutralize}[1]{\expandafter\let\csname c@#1\endcsname\count@}
\makeatother

\title{Time-dependent high-contrast subwavelength resonators}

\author{
	Habib Ammari\thanks{\footnotesize Department of Mathematics, 
		ETH Z\"urich, 
		R\"amistrasse 101, CH-8092 Z\"urich, Switzerland (habib.ammari@math.ethz.ch, erik.orvehed.hiltunen@sam.math.ethz.ch).} \and Erik Orvehed Hiltunen\footnotemark[1]}

\date{}

\begin{document}
	\maketitle
	
	\begin{abstract}
		In the field of metamaterials, many intriguing phenomena arise from having a structure which is periodic in space. In time-dependent structures, conceptually similar properties can arise, which nevertheless have fundamentally different physical implications. In this work, we study time-dependent systems in the context of subwavelength metamaterials. The main result is a capacitance matrix characterization of the band structure, which generalizes previous recent work on static subwavelength metamaterials. \hl{This characterization provides both theoretical insight and efficient numerical methods to compute the dispersion relationship of time-dependent structures}. We exemplify this in several structures exhibiting interesting wave manipulation properties. 
	\end{abstract}
\vspace{0.5cm}
	\noindent{\textbf{Mathematics Subject Classification (MSC2000):} 35J05, 35C20, 35P20.
		
\vspace{0.2cm}

\vspace{0.5cm}	
%

\section{Introduction}
The use of so-called metamaterials has been shown to offer extraordinary usability in controlling and manipulating waves. These effects originate from an intricate, often periodic, spatial structure. A natural generalization of this concept is to consider time-dependent structures, whereby the material parameters depend not only on the spatial variable but are also modulated in time. It is well-known that a ``step-like'' time-modulation (\ie{} an instantaneous shift between two constant values) can cause waves to be reflected and refracted, similarly to sharp spatial interfaces between different materials \cite{mendoncca2002time,koutserimpas2020electromagnetic,bal2019time,morgenthaler1958velocity}. Moreover, time-modulation provides a way to break reciprocity, which is otherwise a fundamental restriction of wave propagation \cite{ourir2019active,guo2019nonreciprocal,sounas2017non,li2019nonreciprocal,nassar2017non}.

The case of static materials which are repeated periodically in space has been well-studied using the Floquet-Bloch theory. In particular, we can define a band structure of the material, which describes the frequency-to-momentum relationship of waves inside the material. Crucially, in space-periodic structures, the wave momentum is contained inside the \emph{Brillouin zone}, and is defined modulo elements of the dual lattice. If there is a gap between the band functions, waves with frequencies inside this band gap cannot propagate through the material and will be exponentially decaying.

Inspired by space-periodic structures, it is natural to consider time-modulations which are periodic. Again we can define a band structure, whose frequencies will now be repeated periodically. An interesting application of this is to create frequency-converting systems, which swap between equivalent frequencies modulo the modulation frequency. Fundamentally, the frequency-conversion is made possible due to the broken energy conservation, which in turn originates from the energy input required to create the time-modulation. Moreover, since there is no energy conservation, there can be unstable waves which are amplified or dampened by the system. In the field of electronics, these ideas have been used to create parametric amplifiers \cite{cullen1958travelling}. Unstable Bloch waves are shown as complex frequencies in the band structure, or (restricting to the real Brillouin zone) as band gaps in the momentum variable (known as \emph{momentum-} or \emph{$k$-gaps})\hlbox{ \cite{koutserimpas2018electromagnetic,martinez2016temporal,zurita2009reflection,cassedy1967dispersion}}.}

Periodic time-modulated structures have been studied in a variety of settings, enabling novel wave phenomena. Due to the energy input, these systems are in general non-Hermitian. This opens the possibility for \emph{exceptional points}, which are parameter points where the eigenmodes coalesce. Such points have a variety of applications, most notably to enhanced sensing \cite{ammari2020highorder,ammari2020exceptional,heiss2012physics}. Moreover, the broken time-reversal symmetry can be used to replicate spin effects from quantum systems. As an example, having phase-shifted (``rotation-like'') modulations can provide a kind of ``artificial spin'' \cite{koutserimpas2018zero,fleury2016floquet,rechtsman2013photonic}. These ideas have been used to study classical analogues of the quantum Hall effect, and so-called Floquet topological insulators \cite{fleury2016floquet,rechtsman2013photonic,raghu2008analogs,nash2015topological,nassar2018quantization,wilson2018temporal,wilson2019temporally}. In \cite{koutserimpas2018zero}, a similar structure was demonstrated to have zero refractive index properties, originating from a linear, cone-like degeneracy in the band structure known as a \emph{Dirac cone} at the origin of the Brillouin zone.

In applications, it is desirable to be able to achieve above mentioned phenomena on \emph{subwavelength} scales. Here, subwavelength means that the length-scale of the system is considerably smaller than the wavelength at which it operates. This is particularly desirable in acoustics, where the wavelengths often are of the order of several meters \cite{ma2016acoustic}. Subwavelength metamaterials can be achieved by having a locally resonant microstructure. In other words, the material is composed of building-blocks which themselves are subwavelength resonators \cite{ammari2017subwavelength,yves2017crystalline,yves2017topological,wang2019subwavelength}.

\emph{High-contrast} resonators are a natural choice of resonators when designing subwavelength metamaterials. Here, the subwavelength nature stems from a high material contrast between the constituting materials of the structure, and such structures can be used to achieve a variety of effects \cite{davies2019fully, ammari2018minnaert,ammari2020exceptional,ammari2020highfrequency,ammari2020highorder,ammari2017subwavelength,ammari2017double,ammari2020honeycomb,MaCMiPaP}. In this work, we study systems of time dependent high-contrast resonators. The goal is to provide a mathematical foundation that explains effects found in time-modulated systems, for waves in the subwavelength frequency regime. We will begin by carefully defining the notion of subwavelength frequencies in time-modulated systems, which due to the frequency conversion is not immediate to interpret. Afterwards, we will characterize the subwavelength band structure of time-modulated systems of high-contrast resonators. \hl{The methods utilized here do not assume weak modulations, and is valid even in the case of large modulation amplitudes.}

This work is structured as follows. In \Cref{sec:problem}, we formulate the general wave equation problem for time-modulated high-contrast resonators. In the following two sections, we study two different types of realizations of this problem. In \Cref{sec:uniform}, we study the case when the time-modulation is applied uniformly in space, both outside and inside the resonators. This can be viewed as a generalization of the problem considered in \cite{koutserimpas2018electromagnetic}, to systems with more general spatial structures. Similarly to \cite{koutserimpas2018electromagnetic}, we find that the time dynamics is governed by the Hill equation, here posed in terms of the instantaneous Minnaert frequencies. We demonstrate that $k$-gaps can be naturally created this way. Although many interesting phenomena can be realized in such systems, it is impossible to create \textit{e.g.} artificial spin using this type of modulation. Therefore, in \Cref{sec:resonatormod}, we consider another time-modulation, applied only to the interior of the resonators. In the asymptotic high-contrast limit, the subwavelength band structure is now characterized by a capacitance matrix formulation, which generalises the case of static systems (see, \textit{e.g.} \cite{ammari2020honeycomb, ammari2017double}). \hl{In the static case, the capacitance matrix formulation offers a rigorous approximation to the differential problem, in terms of a discrete eigenvalue problem.} In the modulated case, this capacitance matrix formulation is now posed as a system of ordinary differential equations in $t$, which can be viewed as a system of coupled Hill equations. \hl{From a computational perspective, we have now reduced the four-dimensional partial differential equation into an ordinary differential equation, allowing efficient numerical methods.} Based on this characterization, we numerically compute the band structure of different materials. This way, we demonstrate the possibility of achieving exceptional points and Dirac cones in the subwavelength regime through the use of time-modulation.

	\section{Problem formulation and preliminary theory}\label{sec:problem}
	In this section, we define the problem of study. Moreover, we introduce the Floquet-Bloch theory for periodic differential equations, \hl{and recall the main results for systems of static resonators.}
	
	\subsection{Problem formulation} \label{sec:formulation}
	We will solve the wave equation in a structure composed of contrasting materials. The material parameter distribution is given by $\rho(x,t)$ and $\kappa(x,t)$. In the example of acoustic waves $\rho$ and $\kappa$ correspond to the density and the bulk modulus of the materials. We emphasize, however, that the equation of study is not restricted to acoustic waves but applies to a wider range of classical wave problems, most notably also to polarized electromagnetic waves.
	
	We study the time-dependent wave equation in dimensions $d=2$ or $d=3$,
	\begin{equation}\label{eq:wave}
		\left(\frac{\p }{\p t } \frac{1}{\kappa(x,t)} \frac{\p}{\p t} - \nabla \cdot \frac{1}{\rho(x,t)} \nabla\right) u(x,t) = 0, \quad x\in \R^d, t\in \R.
	\end{equation}
	Here, $\nabla$ denotes the gradient with respect to $x\in \R^d$. We assume that the geometry is periodic, as illustrated in \Cref{fig:geometry}. Given the linearly independent lattice vectors $l_1, ..., l_d \in \R^d$, we define the lattice $\Lambda$ and the unit cell $Y$ by
	\begin{equation}\Lambda = \{m_1l_1 + ... + m_dl_d \mid m_1, ..., m_d \in \Z\}, \qquad Y = \{a_1l_1 + ... + a_dl_d \mid 0\leq a_1,...,a_d \leq 1\}.\end{equation}
	We assume that each unit cell contains a system of resonators $D\subset Y$. $D$ is constituted by $N$ disjoint domains $D_i$ for $i=1,...,N$, each $D_i$ being connected and having boundary of Hölder class $\p D_i \in C^{1,s}, 0 < s < 1$. The periodic crystal, $\C$, of resonators is then defined by
	\begin{equation}\C = \bigcup_{m\in \Lambda} D+m.\end{equation}
	
	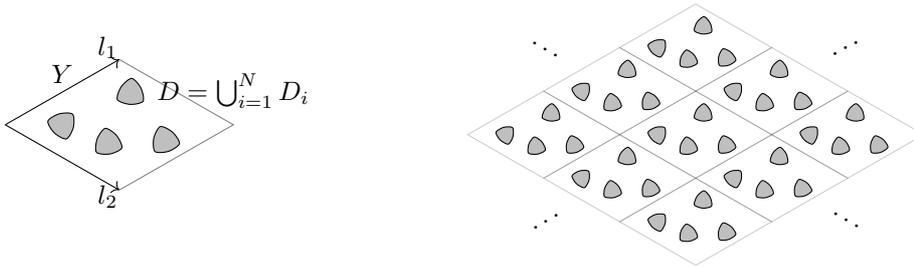
\begin{figure}[tbh]
		\begin{subfigure}[b]{0.48\linewidth}
			\centering
			\begin{tikzpicture}[scale=1.5]
				\begin{scope}[scale=1]
					\pgfmathsetmacro{\rb}{0.13pt}
					\pgfmathsetmacro{\rs}{0.1pt}
					\pgfmathsetmacro{\ao}{326}
					\pgfmathsetmacro{\at}{46}
					\pgfmathsetmacro{\ad}{228}
					\pgfmathsetmacro{\af}{100}
					\coordinate (a) at (1,{1/sqrt(3)});		
					\coordinate (b) at (1,{-1/sqrt(3)});	
					\coordinate (c) at (2,0);
					\draw (2,0.3) node{$D= \bigcup_{i=1}^N D_i$};
					\draw[->] (0,0) -- (a) node[pos=0.9,xshift=0,yshift=7]{ $l_1$} node[pos=0.5,above]{$Y$};
					\draw[->] (0,0) -- (b) node[pos=0.9,xshift=0,yshift=-5]{ $l_2$};
					\draw[opacity=0.5] (a) -- (c) -- (b);
					\begin{scope}[xshift = 1.1cm, yshift=0.28cm,rotate=\ao]
						\draw[fill=lightgray] plot [smooth cycle] coordinates {(0:\rb) (60:\rs) (120:\rb) (180:\rs) (240:\rb) (300:\rs) };
					\end{scope}
					\begin{scope}[xshift = 0.5cm, rotate=\at]
						\draw[fill=lightgray] plot [smooth cycle] coordinates {(0:\rb) (60:\rs) (120:\rb) (180:\rs) (240:\rb) (300:\rs) };
					\end{scope}
					\begin{scope}[xshift = 1.4cm, yshift=-0.14cm, rotate=\ad]
						\draw[fill=lightgray] plot [smooth cycle] coordinates {(0:\rb) (60:\rs) (120:\rb) (180:\rs) (240:\rb) (300:\rs) };
					\end{scope}
					\begin{scope}[xshift = 0.9cm,yshift=-0.16cm, rotate=\af]
						\draw[fill=lightgray] plot [smooth cycle] coordinates {(0:\rb) (60:\rs) (120:\rb) (180:\rs) (240:\rb) (300:\rs) };
					\end{scope}
				\end{scope}
			\end{tikzpicture}
			\vspace{0.65cm}
			\caption{Unit cell $Y$ containing $N$ resonators.}
		\end{subfigure}
		\begin{subfigure}[b]{0.48\linewidth}
			\begin{tikzpicture}[scale=1]
				\begin{scope}[xshift=-5cm,scale=1]
					\coordinate (a) at (1,{1/sqrt(3)});		
					\coordinate (b) at (1,{-1/sqrt(3)});	
					\coordinate (c) at (2,0);
					\pgfmathsetmacro{\rb}{0.13pt}
					\pgfmathsetmacro{\rs}{0.1pt}
					\pgfmathsetmacro{\ao}{326}
					\pgfmathsetmacro{\at}{46}
					\pgfmathsetmacro{\ad}{228}
					\pgfmathsetmacro{\af}{100}
					
					\draw[opacity=0.2] (0,0) -- (a);
					\draw[opacity=0.2] (0,0) -- (b);
					\draw[opacity=0.2] (a) -- (c) -- (b);
					\begin{scope}[xshift = 1.1cm, yshift=0.28cm,rotate=\ao]
						\draw[fill=lightgray] plot [smooth cycle] coordinates {(0:\rb) (60:\rs) (120:\rb) (180:\rs) (240:\rb) (300:\rs) };
					\end{scope}
					\begin{scope}[xshift = 0.5cm, rotate=\at]
						\draw[fill=lightgray] plot [smooth cycle] coordinates {(0:\rb) (60:\rs) (120:\rb) (180:\rs) (240:\rb) (300:\rs) };
					\end{scope}
					\begin{scope}[xshift = 1.4cm, yshift=-0.14cm, rotate=\ad]
						\draw[fill=lightgray] plot [smooth cycle] coordinates {(0:\rb) (60:\rs) (120:\rb) (180:\rs) (240:\rb) (300:\rs) };
					\end{scope}
					\begin{scope}[xshift = 0.9cm,yshift=-0.16cm, rotate=\af]
						\draw[fill=lightgray] plot [smooth cycle] coordinates {(0:\rb) (60:\rs) (120:\rb) (180:\rs) (240:\rb) (300:\rs) };
					\end{scope}

					\begin{scope}[shift = (a)]
						\draw[opacity = 0.2] (0,0) -- (1,{1/sqrt(3)}) -- (2,0) -- (1,{-1/sqrt(3)}) -- cycle; 
						\begin{scope}[xshift = 1.1cm, yshift=0.28cm,rotate=\ao]
							\draw[fill=lightgray] plot [smooth cycle] coordinates {(0:\rb) (60:\rs) (120:\rb) (180:\rs) (240:\rb) (300:\rs) };
						\end{scope}
						\begin{scope}[xshift = 0.5cm, rotate=\at]
							\draw[fill=lightgray] plot [smooth cycle] coordinates {(0:\rb) (60:\rs) (120:\rb) (180:\rs) (240:\rb) (300:\rs) };
						\end{scope}
						\begin{scope}[xshift = 1.4cm, yshift=-0.14cm, rotate=\ad]
							\draw[fill=lightgray] plot [smooth cycle] coordinates {(0:\rb) (60:\rs) (120:\rb) (180:\rs) (240:\rb) (300:\rs) };
						\end{scope}
						\begin{scope}[xshift = 0.9cm,yshift=-0.16cm, rotate=\af]
							\draw[fill=lightgray] plot [smooth cycle] coordinates {(0:\rb) (60:\rs) (120:\rb) (180:\rs) (240:\rb) (300:\rs) };
						\end{scope}
					\end{scope}
					\begin{scope}[shift = (b)]
						\draw[opacity = 0.2] (0,0) -- (1,{1/sqrt(3)}) -- (2,0) -- (1,{-1/sqrt(3)}) -- cycle; 
						\begin{scope}[xshift = 1.1cm, yshift=0.28cm,rotate=\ao]
							\draw[fill=lightgray] plot [smooth cycle] coordinates {(0:\rb) (60:\rs) (120:\rb) (180:\rs) (240:\rb) (300:\rs) };
						\end{scope}
						\begin{scope}[xshift = 0.5cm, rotate=\at]
							\draw[fill=lightgray] plot [smooth cycle] coordinates {(0:\rb) (60:\rs) (120:\rb) (180:\rs) (240:\rb) (300:\rs) };
						\end{scope}
						\begin{scope}[xshift = 1.4cm, yshift=-0.14cm, rotate=\ad]
							\draw[fill=lightgray] plot [smooth cycle] coordinates {(0:\rb) (60:\rs) (120:\rb) (180:\rs) (240:\rb) (300:\rs) };
						\end{scope}
						\begin{scope}[xshift = 0.9cm,yshift=-0.16cm, rotate=\af]
							\draw[fill=lightgray] plot [smooth cycle] coordinates {(0:\rb) (60:\rs) (120:\rb) (180:\rs) (240:\rb) (300:\rs) };
						\end{scope}
					\end{scope}
					\begin{scope}[shift = ($-1*(a)$)]
						\draw[opacity = 0.2] (0,0) -- (1,{1/sqrt(3)}) -- (2,0) -- (1,{-1/sqrt(3)}) -- cycle; 
						\begin{scope}[xshift = 1.1cm, yshift=0.28cm,rotate=\ao]
							\draw[fill=lightgray] plot [smooth cycle] coordinates {(0:\rb) (60:\rs) (120:\rb) (180:\rs) (240:\rb) (300:\rs) };
						\end{scope}
						\begin{scope}[xshift = 0.5cm, rotate=\at]
							\draw[fill=lightgray] plot [smooth cycle] coordinates {(0:\rb) (60:\rs) (120:\rb) (180:\rs) (240:\rb) (300:\rs) };
						\end{scope}
						\begin{scope}[xshift = 1.4cm, yshift=-0.14cm, rotate=\ad]
							\draw[fill=lightgray] plot [smooth cycle] coordinates {(0:\rb) (60:\rs) (120:\rb) (180:\rs) (240:\rb) (300:\rs) };
						\end{scope}
						\begin{scope}[xshift = 0.9cm,yshift=-0.16cm, rotate=\af]
							\draw[fill=lightgray] plot [smooth cycle] coordinates {(0:\rb) (60:\rs) (120:\rb) (180:\rs) (240:\rb) (300:\rs) };
						\end{scope}
					\end{scope}
					\begin{scope}[shift = ($-1*(b)$)]
						\draw[opacity = 0.2] (0,0) -- (1,{1/sqrt(3)}) -- (2,0) -- (1,{-1/sqrt(3)}) -- cycle; 
						\begin{scope}[xshift = 1.1cm, yshift=0.28cm,rotate=\ao]
							\draw[fill=lightgray] plot [smooth cycle] coordinates {(0:\rb) (60:\rs) (120:\rb) (180:\rs) (240:\rb) (300:\rs) };
						\end{scope}
						\begin{scope}[xshift = 0.5cm, rotate=\at]
							\draw[fill=lightgray] plot [smooth cycle] coordinates {(0:\rb) (60:\rs) (120:\rb) (180:\rs) (240:\rb) (300:\rs) };
						\end{scope}
						\begin{scope}[xshift = 1.4cm, yshift=-0.14cm, rotate=\ad]
							\draw[fill=lightgray] plot [smooth cycle] coordinates {(0:\rb) (60:\rs) (120:\rb) (180:\rs) (240:\rb) (300:\rs) };
						\end{scope}
						\begin{scope}[xshift = 0.9cm,yshift=-0.16cm, rotate=\af]
							\draw[fill=lightgray] plot [smooth cycle] coordinates {(0:\rb) (60:\rs) (120:\rb) (180:\rs) (240:\rb) (300:\rs) };
						\end{scope}
					\end{scope}
					\begin{scope}[shift = ($(a)+(b)$)]
						\draw[opacity = 0.2] (0,0) -- (1,{1/sqrt(3)}) -- (2,0) -- (1,{-1/sqrt(3)}) -- cycle; 
						\begin{scope}[xshift = 1.1cm, yshift=0.28cm,rotate=\ao]
							\draw[fill=lightgray] plot [smooth cycle] coordinates {(0:\rb) (60:\rs) (120:\rb) (180:\rs) (240:\rb) (300:\rs) };
						\end{scope}
						\begin{scope}[xshift = 0.5cm, rotate=\at]
							\draw[fill=lightgray] plot [smooth cycle] coordinates {(0:\rb) (60:\rs) (120:\rb) (180:\rs) (240:\rb) (300:\rs) };
						\end{scope}
						\begin{scope}[xshift = 1.4cm, yshift=-0.14cm, rotate=\ad]
							\draw[fill=lightgray] plot [smooth cycle] coordinates {(0:\rb) (60:\rs) (120:\rb) (180:\rs) (240:\rb) (300:\rs) };
						\end{scope}
						\begin{scope}[xshift = 0.9cm,yshift=-0.16cm, rotate=\af]
							\draw[fill=lightgray] plot [smooth cycle] coordinates {(0:\rb) (60:\rs) (120:\rb) (180:\rs) (240:\rb) (300:\rs) };
						\end{scope}
					\end{scope}
					\begin{scope}[shift = ($-1*(a)-(b)$)]
						\draw[opacity = 0.2] (0,0) -- (1,{1/sqrt(3)}) -- (2,0) -- (1,{-1/sqrt(3)}) -- cycle; 
						\begin{scope}[xshift = 1.1cm, yshift=0.28cm,rotate=\ao]
							\draw[fill=lightgray] plot [smooth cycle] coordinates {(0:\rb) (60:\rs) (120:\rb) (180:\rs) (240:\rb) (300:\rs) };
						\end{scope}
						\begin{scope}[xshift = 0.5cm, rotate=\at]
							\draw[fill=lightgray] plot [smooth cycle] coordinates {(0:\rb) (60:\rs) (120:\rb) (180:\rs) (240:\rb) (300:\rs) };
						\end{scope}
						\begin{scope}[xshift = 1.4cm, yshift=-0.14cm, rotate=\ad]
							\draw[fill=lightgray] plot [smooth cycle] coordinates {(0:\rb) (60:\rs) (120:\rb) (180:\rs) (240:\rb) (300:\rs) };
						\end{scope}
						\begin{scope}[xshift = 0.9cm,yshift=-0.16cm, rotate=\af]
							\draw[fill=lightgray] plot [smooth cycle] coordinates {(0:\rb) (60:\rs) (120:\rb) (180:\rs) (240:\rb) (300:\rs) };
						\end{scope}
					\end{scope}
					\begin{scope}[shift = ($(a)-(b)$)]
						\draw[opacity = 0.2] (0,0) -- (1,{1/sqrt(3)}) -- (2,0) -- (1,{-1/sqrt(3)}) -- cycle; 
						\begin{scope}[xshift = 1.1cm, yshift=0.28cm,rotate=\ao]
							\draw[fill=lightgray] plot [smooth cycle] coordinates {(0:\rb) (60:\rs) (120:\rb) (180:\rs) (240:\rb) (300:\rs) };
						\end{scope}
						\begin{scope}[xshift = 0.5cm, rotate=\at]
							\draw[fill=lightgray] plot [smooth cycle] coordinates {(0:\rb) (60:\rs) (120:\rb) (180:\rs) (240:\rb) (300:\rs) };
						\end{scope}
						\begin{scope}[xshift = 1.4cm, yshift=-0.14cm, rotate=\ad]
							\draw[fill=lightgray] plot [smooth cycle] coordinates {(0:\rb) (60:\rs) (120:\rb) (180:\rs) (240:\rb) (300:\rs) };
						\end{scope}
						\begin{scope}[xshift = 0.9cm,yshift=-0.16cm, rotate=\af]
							\draw[fill=lightgray] plot [smooth cycle] coordinates {(0:\rb) (60:\rs) (120:\rb) (180:\rs) (240:\rb) (300:\rs) };
						\end{scope}
					\end{scope}
					\begin{scope}[shift = ($-1*(a)+(b)$)]
						\draw[opacity = 0.2] (0,0) -- (1,{1/sqrt(3)}) -- (2,0) -- (1,{-1/sqrt(3)}) -- cycle; 
						\begin{scope}[xshift = 1.1cm, yshift=0.28cm,rotate=\ao]
							\draw[fill=lightgray] plot [smooth cycle] coordinates {(0:\rb) (60:\rs) (120:\rb) (180:\rs) (240:\rb) (300:\rs) };
						\end{scope}
						\begin{scope}[xshift = 0.5cm, rotate=\at]
							\draw[fill=lightgray] plot [smooth cycle] coordinates {(0:\rb) (60:\rs) (120:\rb) (180:\rs) (240:\rb) (300:\rs) };
						\end{scope}
						\begin{scope}[xshift = 1.4cm, yshift=-0.14cm, rotate=\ad]
							\draw[fill=lightgray] plot [smooth cycle] coordinates {(0:\rb) (60:\rs) (120:\rb) (180:\rs) (240:\rb) (300:\rs) };
						\end{scope}
						\begin{scope}[xshift = 0.9cm,yshift=-0.16cm, rotate=\af]
							\draw[fill=lightgray] plot [smooth cycle] coordinates {(0:\rb) (60:\rs) (120:\rb) (180:\rs) (240:\rb) (300:\rs) };
						\end{scope}
					\end{scope}
					\begin{scope}[shift = ($2*(a)$)]
						\draw (1,0) node[rotate=30]{$\cdots$};
					\end{scope}
					\begin{scope}[shift = ($-2*(a)$)]
						\draw (1,0) node[rotate=210]{$\cdots$};
					\end{scope}
					\begin{scope}[shift = ($2*(b)$)]
						\draw (1,0) node[rotate=-30]{$\cdots$};
					\end{scope}
					\begin{scope}[shift = ($-2*(b)$)]
						\draw (1,0) node[rotate=150]{$\cdots$};
					\end{scope}
				\end{scope}
			\end{tikzpicture}
			\caption{Infinite, periodic system with unit cell $Y$ and lattice $\Lambda$.}
		\end{subfigure}
		\caption{Example illustrations of the unit cell and the infinite system of resonators.} \label{fig:geometry}
	\end{figure}
	
	\hl{The material parameters $\kappa$ and $\rho$ are assumed to be piecewise constant in $x$. The corresponding case without time-modulation has previously been extensively studied }\cite{ammari2018minnaert,ammari2020honeycomb,ammari2017subwavelength}, where $\rho$ and $\kappa$ satisfy
	\begin{equation}\label{eq:static}
		\kappa(x,t) =  \begin{cases}
			\kappa_0, & x \in \R^d \setminus \overline{\C}, \\ \kappa_\r, & x\in \C,
		\end{cases} \qquad \rho(x,t) =  \begin{cases}
			\rho_0, & x \in \R^d \setminus \overline{\C}, \\ \rho_\r, & x\in \C.
		\end{cases}\end{equation}
	We define the contrast parameter $\delta$ and the wave speeds $v_0, v_\r$ as
	\begin{equation}\delta = \frac{\rho_\r}{\rho_0}, \quad v_0 = \sqrt{\frac{\kappa_0}{\rho_0}}, \quad v_\r = \sqrt{\frac{\kappa_\r}{\rho_\r}}.\end{equation}	
	\hl{In this setting, subwavelength resonance requires a high material contrast in $\rho$:}
	\begin{equation}\delta \ll 1, \quad v_0,v_\r = O(1).\end{equation}	
	In this case, there are resonant frequencies $\omega$ (known as \emph{subwavelength} resonant frequencies) satisfying $\omega \to 0$ as $\delta \to 0$ (typically, in the problems considered here, $\omega$ scales as $\omega=O(\delta^{1/2})$).
	
	In this work, we will extend this theory to the time-dependent case when the material parameters $\kappa$ and $\rho$, in addition to being piecewise constant in $x$, also depend periodically on $t$ with frequency $\Omega$. Again, in order to achieve subwavelength resonance, we assume that $\rho$ is much smaller inside $\C$ compared to the outside:
	\begin{equation}\mathhl{\frac{\rho(x,t)}{\rho(y,t)} = O(\delta), \quad \text{for all }  x \in \C, \ y \in \R^d\setminus \overline{\C}, \ t\in \R,}\end{equation}
	where $\delta \ll 1$ is the high-contrast parameter. \hl{We will consider general $\Omega$, where, typically, the most interesting regime is $\Omega = O(\delta^{1/2})$ (which corresponds to the case when $\Omega$ has the same asymptotic behaviour as the resonances $\omega$). The current setting is designed in order to enhance the subwavelength nature of phenomena due to time-modulation. Related structures have been previously proposed and implemented, where the time-modulated material parameters are mainly achieved through controlling the resonator boundaries} \cite{psiachos2021band,sugino2020nonreciprocal,xu2020physical}.
	
	\subsection{Floquet-Bloch theory and quasiperiodic layer potentials}\label{sec:floquet}
	In this section we give a brief introduction to the Floquet-Bloch theory (for further details we refer, for example, to \cite{kuchment}), which is the typical technique used to study differential equations with periodic coefficients. We begin by outlining the theory in the case of ordinary systems of differential equations in the time variable $t$, 
	\begin{equation} \label{eq:ODE}
		y'(t) = A(t)y(t), \quad t\in \R,
	\end{equation}
	for some $N\times N$ matrix function $A(t)$ which is $T$-periodic and piecewise continuous in $t$. If $Y(t)$ denotes the (matrix-valued) fundamental solution, then Floquet's theorem states that there is a constant matrix $B$ such that 
	\begin{equation}Y(t) = e^{iBt}P(t),\end{equation}
	for some $T$-periodic matrix function $P$. For each eigenvalue $e^{\iu\omega}$ of $e^{\iu B}$, there is a \emph{Bloch solution} $y(t)$ satisfying the $\omega$-quasiperiodicity condition, \ie{} that $y(t)e^{-\iu \omega t}$ is $T$-periodic. If $e^{iB}$ is a diagonalizable matrix, there is a basis of Bloch solutions to \eqref{eq:ODE}. Observe that $\omega$, which we refer to as a \emph{quasifrequency}, is defined modulo $\Omega = \frac{2\pi}{T}$. Therefore, we define the (time-) Brillouin zone $\omega \in Y^*_t := \mathbb{C} / (\Omega \Z)$. Observe that we allow complex quasifrequencies. If $\omega$ is real, any $\omega$-quasiperiodic function $y(t)$ is bounded in $t$ and is said to be \emph{stable}.
	
 	Next, we define analogous concepts in higher dimensions. Given the lattice $\Lambda$ as defined above, a function $f(x)\in L^2(\R^d)$ is $\alpha$-quasiperiodic if $e^{-\iu \alpha \cdot x}f(x)$ is a $\Lambda$-periodic function of $x$. The quasiperiodicity (or \emph{quasimomentum}) $\alpha$ is defined modulo elements of the \emph{dual lattice} $\Lambda^*$, which is the lattice generated by the dual vectors $\alpha_1,...,\alpha_d$ defined through
	\begin{equation}\alpha_i\cdot l_j = 2\pi \delta_{i,j}, \quad i,j=1,...,d.\end{equation}
	The (space-) \textit{Brillouin zone} $Y^*$ is defined as the torus $Y^*:= \R^d /\Lambda^*$. Given a function $f\in L^2(\R^d)$, the Floquet transform is defined as
	\begin{equation}
		\F[f](x,\alpha) := \sum_{m\in \Lambda} f(x-m) e^{\iu \alpha\cdot m}.
	\end{equation}
	$\F[f]$ is always $\alpha$-quasiperiodic in $x$ and periodic in $\alpha$. The Floquet transform is an invertible map $\F:L^2(\R^d) \rightarrow L^2(Y\times Y^*)$ with inverse given by (see, for instance, \cite{MaCMiPaP,kuchment})
	\begin{equation*}
		\F^{-1}[g](x) = \frac{|Y|}{2\pi}\int_{Y^*} g(x,\alpha) \dx \alpha, \quad x\in \R^d.
	\end{equation*}
	
	\hl{Applying the Floquet transform to the wave equation }\eqref{eq:wave} in $x$, \hl{and seeking quasiperiodic solutions in $t$, we obtain the differential problem}
 	\begin{equation} \label{eq:wave_transf}
 	\begin{cases}\ \ds \left(\frac{\p }{\p t } \frac{1}{\kappa(x,t)} \frac{\p}{\p t} - \nabla \cdot \frac{1}{\rho(x,t)} \nabla\right) u(x,t) = 0,\\[0.3em]
 		\	u(x,t)e^{-\iu \alpha\cdot x} \text{ is $\Lambda$-periodic in $x$,}\\
 		\	u(x,t)e^{-\iu \omega t} \text{ is $T$-periodic in $t$}. 
 	\end{cases}
 	\end{equation} 
 	For a given $\alpha\in Y^*$, we seek $\omega\in Y_t^*$ such that there is a non-zero solution $u$ to \eqref{eq:wave_transf}. Such quasifrequencies $\omega(\alpha)$, as functions of $\alpha$, are known as \emph{band functions} and together constitute the band structure, or dispersion relationship, of the material. Observe that the band functions depend continuously on the high-contrast parameter $\delta$.
 	 	
 	We remark that the quasifrequencies as defined above can indeed be observed as a generalization of the usual concept of frequency in the unmodulated case. If both $\kappa$ and $\rho$ are constant in $t$, we can not define a minimal periodicity $T$. Requiring \eqref{eq:wave_transf} to hold for any $T$, we have that $u(x,t)$ is a time-harmonic wave $u(x,t) = v(x)e^{\iu \omega t}$ with (ordinary) frequency $\omega$. 
 	
 	For each $n\in \Z^ +$, we can represent the (time-) Brillouin zone $Y^*_t$ by the $n$\textsuperscript{th} Brillouin zone $Y^{*,n}_t$, which is an union of two strips:
 	\begin{equation}Y^{*,n}_t=\left[-\frac{n\Omega}{2},-\frac{(n-1)\Omega}{2}\right)\times \iu \R \cup \left[\frac{(n-1)\Omega}{2},\frac{n\Omega}{2}\right) \times \iu \R.\end{equation}
 	We can think of the collection of the $n$\textsuperscript{th} Brillouin zones as a lifting of the Brillouin zone $Y^*_t$ to $\mathbb{C}$. Since $e^{- \iu \omega t}u(x,t) $ is $T$-periodic in $t$ we have a Fourier series expansion as 
 	\begin{equation}u(x,t)= e^{\iu \omega t}\sum_{n = -\infty}^\infty v_n(x)e^{\iu n\Omega t}.\end{equation}
 	Although the quasifrequencies $\omega$ are defined modulo $\Omega\Z$, choosing a different representation of $\omega$ amounts to a shifting of the Fourier coefficients $v_n(x)$. This provides a way to associate, at least intuitively, each $\omega$ to a certain $n$\textsuperscript{th} Brillouin zone, where $n$ is chosen (in some sense) to minimise the oscillations of the Fourier series part of $u$.	This idea will be utilized in a precise manner in \Cref{def:sub} below.
 	
 	Due to the periodic nature of $Y_t^*$, the usual definition of subwavelength frequencies does not apply to quasifrequencies. For example, in the particular case when $\Omega = O(\delta^{1/2})$ (which will be of interest later on), the whole Brillouin zone scales as $O(\delta^{1/2})$ but will typically contain an infinite number of quasifrequencies which originate from folding of non-subwavelength frequencies. Due to this, we introduce the following definition.
 	\begin{defn} \label{def:sub}
 		A quasifrequency $\omega = \omega(\delta) \in Y^*_t$ of \eqref{eq:wave_transf} is said to be a \emph{subwavelength quasifrequency} if there is a corresponding Bloch solution $u(x,t)$, depending continuously on $\delta$, which can be written as
 		\begin{equation}u(x,t)= e^{\iu \omega t}\sum_{n = -\infty}^\infty v_n(x)e^{\iu n\Omega t},\end{equation}
 		where 
 		\begin{equation}\omega \rightarrow 0 \ \text{and} \ M\Omega \rightarrow 0 \ \text{as} \ \delta \to 0,\end{equation}
 		for some integer-valued function $M=M(\delta)$ such that, as $\delta \to 0$, we have
 		\begin{equation}\sum_{n = -\infty}^\infty \|v_n\|_{L^2(Y)} = \sum_{n = -M}^M \|v_n\|_{L^2(Y)} + o(1).\end{equation}
 	\end{defn}
 	In other words, a quasifrequency is a subwavelength quasifrequency if the corresponding Bloch solution only contains components which are either very small, or are in the subwavelength frequency regime. In the static case, we can choose $M=0$ and obtain the usual definition of a subwavelength frequency.
 	
 	\subsection{Capacitance matrix formulation of the static problem}\label{sec:cap_static}
 	\hl{Before analysing the time-modulated problem, we briefly review the subwavelength resonance of the static problem with constant parameters as defined in} \eqref{eq:static} \cite{ammari2018minnaert,ammari2020honeycomb, ammari2017subwavelength}.\hl{Considering time-harmonic solutions $u(x,t) = e^{\iu\omega t}v(x)$, we have from }\eqref{eq:wave_transf} \hl{that $v$ satisfies}
 	
 	\begin{equation} \label{eq:space}
 		\left\{
 		\begin{array} {ll}
 			\ds \Delta {v}+ \frac{\lambda\rho_0}{\kappa_0} {v}  = 0 & \text{in } \R^d \setminus \overline{\C}, \\[0.3em]
 			\ds \Delta {v}+\frac{\lambda\rho_\r}{\kappa_\r}{v}  = 0 & \text{in } \C, \\
 			\nm
 			\ds  {v}|_{+} -{v}|_{-}  = 0  & \text{on } \partial \C, \\
 			\nm
 			\ds  \delta \frac{\partial {v}}{\partial \nu} \bigg|_{+} - \frac{\partial {v}}{\partial \nu} \bigg|_{-} = 0 & \text{on } \partial \C, \\
 			v(x)e^{-\iu \alpha\cdot x} \ &\text{is $\Lambda$-periodic},
 		\end{array}
 		\right.
 	\end{equation}
 	where $\lambda= \omega^2$ and \hl{$|_{+,-}$ denote the limits from outside and inside $\C$, respectively.} Here, we have interpreted \eqref{eq:wave} in a weak sense, which leads to the so-called \emph{transmission conditions} posed on $\p \C$ in \eqref{eq:space}. We denote the band functions (or \emph{resonant frequencies}) of \eqref{eq:space} by $\omega^\alpha_{\mathrm{s},i}$ for $i=1,2,...$ (here, subscript $\mathrm{s}$ is short for \emph{static}). In other words,  \eqref{eq:space} admits at $\lambda = \left(\omega^\alpha_{\mathrm{s},i}\right)^2$ a non-zero solution. It is well-known that the first $2N$ band functions scale as $O(\delta^{1/2})$ and are thereby in the subwavelength regime for small $\delta$. $N$ of these band functions have positive real part. Next, we use the capacitance matrix formulation to get explicit asymptotic expansions of these subwavelength band functions when $\delta$ is small.
 	
 	Define the $\alpha$-quasiperiodic Green's function $G^{\alpha,k}(x,y)$ to satisfy
 	\begin{equation*}
 		\Delta_xG^{\alpha,k}(x,y) + k^2G^{\alpha,k}(x,y) = \sum_{n \in \Lambda} \delta(x-n)e^{\iu\alpha\cdot n}.
 	\end{equation*}
 	If $k \neq |\alpha+q|$ for all $q\in \Lambda^*$, it can be shown \cite{MaCMiPaP,ammari2009layer} that $G^{\alpha,k}$ is given by 
 	\begin{equation*}
 		G^{\alpha,k}(x,y)= \frac{1}{|Y|}\sum_{q\in \Lambda^*} \frac{e^{\iu(\alpha+q)\cdot (x-y)}}{ k^2-|\alpha+q|^2}.
 	\end{equation*}
 	Let $D\subset \R^d$ be as in \Cref{sec:formulation}. We define the quasiperiodic single layer potential $\mathcal{S}_D^{\alpha,k}: L^2(\partial D) \rightarrow H_{\textrm{loc}}^1(\R^d)$ by
 	\begin{equation}\mathcal{S}_D^{\alpha,k}[\phi](x) := \int_{\partial D} G^{\alpha,k} (x,y) \phi(y) \dx\sigma(y),\quad x\in \mathbb{R}^d.\end{equation}
 	Here, the space $H_{\textrm{loc}}^1(\R^d)$ consists of functions that are square integrable and with a square integrable weak first derivative, on every compact subset of $\R^d$. Taking the trace on $\p D$, it is well-known that $\mathcal{S}_D^{\alpha,0} : L^2(\p D) \rightarrow H^1(\p D)$ is invertible if $\alpha \neq  0$ \cite{MaCMiPaP}. Moreover, on the boundary $\p D$, $\mathcal{S}_D^{\alpha,k}$ satisfies the so-called \emph{jump relations}
 	\begin{equation} \label{eq:jump1}
 		\mathcal{S}_D^{\alpha,k}[\phi]\big|_+ = \mathcal{S}_D^{\alpha,k}[\phi]\big|_-,
 	\end{equation}
 	and
 	\begin{equation} \label{eq:jump2}
 		\frac{\p}{\p\nu} \mathcal{S}_D^{\alpha,k}[\phi] \Big|_{\pm} = \left( \pm \frac{1}{2} I +( \mathcal{K}_D^{-\alpha,k} )^*\right)[\phi],
 	\end{equation}
 	where  $I$ is the identity operator, $\partial/\partial \nu_x$ denotes the outward normal derivative at $x\in\p D$. Moreover, $(\mathcal{K}_D^{-\alpha,k})^*: L^2(\partial D) \rightarrow L^2(\partial D)$ is the quasiperiodic Neumann-Poincaré operator given by
 	\begin{equation} (\mathcal{K}_D^{-\alpha, k} )^*[\phi](x):= \int_{\p D} \frac{\p}{\p\nu_x} G^{\alpha,k}(x,y) \phi(y) \dx\sigma(y).\end{equation}
 	For low frequencies, \ie{} as $k \to0$, we have the asymptotic expansions
 	\begin{equation}\label{eq:Sexp}
 		\mathcal{S}_D^{\alpha,k} = \mathcal{S}_D^{\alpha,0} + O(k^2),
 	\end{equation}
 	and 
 	\begin{equation}\label{eq:expK}
 		(\mathcal{K}_D^{-\alpha,k})^* = (\mathcal{K}_D^{-\alpha,0})^* + O(k^2).
 	\end{equation}
 	Moreover, we have the following well-known integration formula (see, for example, \cite{ammari2020honeycomb}),
 	\begin{equation}\label{eq:intK}
 		\int_{\p D_i} \left(\frac{1}{2}I + (\mathcal{K}_D^{-\alpha,0})^*\right)[\phi] \dx \sigma = \int_{\p D_i} \phi \dx \sigma.
 	\end{equation}
 	We are now ready to state the capacitance formulation of the problem. For $\alpha \neq 0$, the basis functions $\psi_i^\alpha$ and the capacitance coefficients $C_{ij}^\alpha$ are defined as
 	\begin{equation}\label{eq:psiC}
 		\psi_i^\alpha = \left(\mathcal{S}_D^{\alpha,0}\right)^{-1}[\chi_{\p D_i}], \qquad C_{ij}^\alpha= -\int_{\p D_i} \psi_j^\alpha   \dx \sigma,
 	\end{equation}
 	for $i,j=1,...,N$. The capacitance matrix $C^\alpha$ is defined as the matrix $C^\alpha = \left(C_{ij}^\alpha\right)$. For simplicity we state the result when all resonators have equal volume. We then have the following result \cite{ammari2017subwavelength,ammari2020honeycomb}.
 	\begin{thm} \label{thm:static}
 		\hl{
 		For $|\alpha|\neq0$, the band functions $\omega^\alpha_{\mathrm{s},i},~i=1,...,N$ can be approximated as}
 		\begin{equation}
 			\omega^\alpha_{\mathrm{s},i}= \sqrt{\frac{\delta \lambda_i^\alpha }{|D_1|}}  v_\r + O(\delta),
 		\end{equation}
 		\hl{where $|D_1|$ is the volume of one resonator and $\lambda_i^\alpha,~i=1,...,N$ are the eigenvalues of the capacitance matrix $C^\alpha$.}
 	\end{thm}
 	\hl{The capacitance formulation provides a method to integrate} \eqref{eq:space} \hl{which is posed on the spatial domain. In particular, for small $\delta$, the solutions to} \eqref{eq:space} \hl{are approximately constant inside each resonator. Therefore, the continuous differential equation reduces to a discrete equation where the solutions are determined by these constant values (which are given by the eigenvectors of $C^\alpha$). For more details, we refer to} \cite{ammari2018minnaert,ammari2017subwavelength,ammari2020honeycomb}.
 	
	\section{Uniformly modulated high-contrast resonators} \label{sec:uniform}
	In this section, we study the case when the time-modulation is applied uniformly in space. In other words, the modulation occurs both outside and inside the resonators with some envelopes $\kappa_t$, $\rho_t$ as follows:
	\begin{equation}\label{eq:uniform}
	\kappa(x,t) = \kappa_x(x)\kappa_t(t), \qquad \rho(x,t) = \rho_x(x)\rho_t(t),
	\end{equation}
	where the factors $\kappa_x$ and  $\rho_x$ describe the spatial parts and are defined as 
	\begin{equation}\label{eq:param_crystal}
	\kappa_x(x) =  \begin{cases}
	\kappa_0, & x \in \R^d \setminus \overline{\C}, \\ \kappa_\r, & x\in \C,
	\end{cases} \qquad \rho_x(x) =  \begin{cases}
	\rho_0, & x \in \R^d \setminus \overline{\C}, \\ \rho_\r, & x\in \C.
	\end{cases}
	\end{equation}
	We assume that $\rho_t$ is piecewise continuous and that $\kappa_t\in C^1(\R)$.
	We seek solutions to \eqref{eq:wave_transf} by separation of variables. Writing $u(x,t) = \Phi(t) v(x)$, we find that
	\begin{equation}\label{eq:time}
		\left\{
		\begin{array} {ll}
		\ds \frac{\dx }{\dx t}\frac{1}{\kappa_t(t)} \frac{\dx}{\dx t}\Phi(t) +\frac{\lambda}{\rho_t(t)}\Phi(t) = 0, \\[0.3em]
		\ds \Phi(t)e^{-\iu \omega t} \ \text{ is $T$-periodic},
		\end{array}
	\right.
	\end{equation}
	and that the spatial part $v$ satisfies \eqref{eq:space}, \ie{} the same equation as in the static case. Substituting $\Phi(t) = \sqrt{\kappa_t(t)} \Psi(t)$ in \eqref{eq:time}, we obtain the following result.

	\begin{prop} \label{prop:infinite}
		Assume that the material parameters are given by \eqref{eq:uniform} and  \eqref{eq:param_crystal}. Then, the quasifrequencies $\omega = \omega(\alpha) \in Y^*_t$ to the wave equation \eqref{eq:wave} are given by the quasifrequencies of the equation
		\begin{equation}\Psi''(t) +  \left(\left(\omega^\alpha_i(t)\right)^2 + \frac{\sqrt{\kappa_t}}{2}\frac{\dx}{\dx t} \frac{\kappa_t'}{\kappa_t^{3/2}}\right)\Psi(t) = 0,\end{equation}
		for $i=1,2,...$. Here, $\omega_i^\alpha(t)$ are the instantaneous resonant frequencies defined by $\omega_i^ \alpha(t) = \omega^\alpha_{\mathrm{s},i}\sqrt{\frac{\kappa_t(t)}{\rho_t(t)}}$ for $\alpha\in Y^*$.
	\end{prop}
	
	\begin{rmk}
		\hl{Since the spatial part $v$ satisfies the same equation as in the static case, }\eqref{eq:time} \hl{reveals the main effects due to this uniform modulation. Instead of being time-harmonic solutions, the eigenmodes $v(x)$ have time dynamics which are governed by the solutions $\Phi(t)$ to equation} \eqref{eq:time}.
	\end{rmk}
	
	\begin{rmk}
		The ordinary differential equation in \Cref{prop:infinite} is a type of \emph{Hill equation}. The proposition shows that the band structure of the modulated system is specified by a Hill equation in terms of the band structure of the static system. In particular, the static and modulated systems will have the same degeneracies (modulo $\Omega$). This limits the range of possible phenomena that can be induced with the uniform modulation specified in \eqref{eq:uniform} and \eqref{eq:param_crystal}. In \Cref{sec:resonatormod} we consider a different type of modulation which can induce richer phenomena. 
	\end{rmk}

	\begin{rmk}\label{rmk:impedance},
		\hl{In the case when the wave impedance is constant in $t$, the Hill equation is easily solved. In the current framework, the wave impedance is given by $\sqrt{\kappa\rho}$, so a constant impedance is equivalent to $\rho_t(t) = 1/\kappa_t(t)$. If we let $\lambda = (\omega_{\mathrm{s},i}^\alpha)^2$, we can simplify} \eqref{eq:time} \hl{by substitution as follows:}
		\begin{equation}\label{eq:diagonal}
			\frac{\dx }{\dx t} \Psi(t) = \iu \omega_{\mathrm{s},i}^\alpha \kappa_t(t) \begin{pmatrix} 1 & 0 \\ 0 & -1\end{pmatrix} \Psi(t), \quad \text{where} \quad
			\Psi = \begin{pmatrix} \iu\omega_{\mathrm{s},i}^\alpha \Phi + \frac{1}{\kappa_t(t)} \frac{\dx \Phi}{\dx t} \\[0.3em] \iu\omega_{\mathrm{s},i}^\alpha \Phi - \frac{1}{\kappa_t(t)} \frac{\dx \Phi}{\dx t} \end{pmatrix}.
		\end{equation}
		\hl{Then the quasifrequencies $\omega$ of} \eqref{eq:diagonal} \hl{are given by} \begin{equation}
			\omega = \omega_{\mathrm{s},i}^\alpha\widetilde\kappa, \qquad \widetilde \kappa = \frac{1}{T}\int_{0}^{T}\kappa_t(t)\dx t.
		\end{equation}
		\hl{In other words, the quasifrequencies of the modulated system coincide with those of the static, averaged, system. Any phenomena emerging from the time-modulation, \textit{e.g.} $k$-gaps, therefore cannot occur in this special case} \cite{koutserimpas2018electromagnetic,martinez2017standing,martinez2016temporal}.
	\end{rmk}
	
	\subsection{Sinusoidal time-modulation}
	In this section, we study consequences of \Cref{prop:infinite} in the case when the modulation is sinusoidal with some amplitude $\varepsilon$ and frequency $\Omega$.
	
	\subsubsection{Modulation of $\rho$} \label{sec:rhomod}
	We begin by studying the case
	\begin{equation}\kappa_t(t) = 1, \qquad \rho_t(t) = \frac{1}{1 + \varepsilon\cos(\Omega t)}, \quad 0 \leq \varepsilon < 1.\end{equation}
	Then, setting 
	\begin{equation}\tau = \frac{\Omega t}{2}, \qquad a = \left(\frac{2\omega_{\mathrm{s},i}^\alpha}{\Omega}\right)^2, \qquad q = -2\varepsilon\left(\frac{\omega_{\mathrm{s},i}^\alpha}{\Omega}\right)^2,\end{equation}
	we obtain the equation
	\begin{equation}\label{eq:mathieu}
	\phi''(\tau) + (a-2q\cos(2\tau))\phi(\tau) = 0,
	\end{equation}
	where $\phi$ is defined as $\phi(\tau) = \Psi(t)$. Equation \eqref{eq:mathieu} is known as \emph{Mathieu's equation}, and the quasifrequencies $\nu$ of this equation are known as the \emph{characteristic exponents}. Observe that $\nu$ and $\omega$ are related as
	\begin{equation}\omega = \frac{\Omega\nu}{2}.\end{equation}
	In the limit as $\delta \rightarrow 0$, the behaviour is fundamentally dependent on how $\varepsilon$ and $\Omega$ scale with $\delta$. In particular, some cases allow complex band functions (\ie{} unstable solutions, or $k$-gaps). Crucially, we always have $|q| < a/2$, which means that complex bands occur around points where $\omega_{\mathrm{s},i}^\alpha$ is close to $\frac{n\Omega}{2}, n \in \Z$ \cite{mclachlan1951theory}. As $\delta \to 0$, the subwavelength quasifrequencies are the quasifrequencies associated to the subwavelength static frequencies $\omega_{\mathrm{s},i}^\alpha$ which scale as $O(\delta^{1/2})$.
	
	\smallskip
	
	\noindent \textbf{Case (i): $\varepsilon=o(1)$ as $\delta \rightarrow 0$.}
	\smallskip
	
	In this case, we have that $q/a = o(1)$. From the theory of Mathieu's equation, we have that \cite{mclachlan1951theory}
	\begin{equation}\nu = \pm\sqrt{a}\big(1+o(1)\big).\end{equation}
	It follows that, to leading order, the band structure $\omega(\alpha)$ coincides with the unmodulated case:  $\omega(\alpha) = \omega_{\mathrm{s},i}^\alpha + o(1)$.
	
	\smallskip
	
	\noindent \textbf{Case (ii): $\Omega=O(\delta^{1/2})$ and $\varepsilon$ is fixed as $\delta \rightarrow 0$.}
	\smallskip
	
	In this case, we have that $\Omega$ and $\omega_{\mathrm{s},i}^\alpha$ have the same scaling in $\delta$, and so $a,q=O(1)$. Choosing $\Omega$ within the band region of the static problem, we expect complex band functions. Notably, to create a subwavelength $k$-gap, we require all subwavelength quasifrequencies to be complex. In the case $N=1$, \ie{} a single resonator inside the unit cell, both subwavelength bands will open into $k$-gaps around the same point (as in \Cref{fig:SLc} below). In the case of multiple resonators, $N>1$, a $k$-gap can be achieved by centring the complex frequencies around a degeneracy of the static problem (as in \Cref{fig:HLc} below).
	
	\smallskip
	
	\noindent \textbf{Case (iii): $\Omega$ and $\varepsilon$ are fixed as $\delta \rightarrow 0$.}
	\smallskip
	
	In this case we have $a,q=O(\delta)$, which implies that $\nu = O(\delta^{1/2})$. Consequently, there will be subwavelength quasifrequencies $\omega = O(\delta^{1/2})$, even though $\Omega$ is constant. For $a,q$ around $0$ and $a>0$, the Bloch solutions are stable \cite{mclachlan1951theory} and there are no $k$-gaps in this case.	
	
	\begin{rmk}
		We can alternatively study the case of step-like changes in the material parameters, where $\rho_t$ changes between the values $\rho_1$ and $\rho_2$ at time $t_0$. Instead of the Mathieu equation, we obtain in this case the Meissner equation,
		\begin{equation}\Psi''(t) + \left(\omega^\alpha_i(t)\right)^2\Psi(t) = 0, \qquad  \omega^\alpha_i(t) = \begin{cases}\ds \frac{\omega_{\mathrm{s},i}^\alpha}{\sqrt{\rho_1}}, \quad & -\frac{\Omega}{2} < t < t_0, \\[0.7em] \ds \frac{\omega_{\mathrm{s},i}^\alpha}{\sqrt{\rho_2}}, & t_0 < t < \frac{\Omega}{2},\end{cases}\end{equation} analogously to \cite{koutserimpas2018electromagnetic} but now posed in terms of the instantaneous resonant frequencies $\omega_i^\alpha(t)$. \hl{Qualitatively, the behaviour in this case is similar to the Mathieu equation studied before.}
	\end{rmk}
	
	\subsubsection{Modulation of $\rho$ and $\kappa$}
	Here, we briefly mention the case when sinusoidal modulation is applied to both $\rho$ and $\kappa$. We assume
	\begin{equation}\kappa_t(t) = \frac{1}{1 + \varepsilon\cos(\Omega t)}, \qquad \rho_t(t) = \frac{1}{1 + \varepsilon\cos(\Omega t)}, \quad 0 \leq \varepsilon < 1.\end{equation}
	From \Cref{prop:infinite} we then find
	\begin{equation}\Psi'' + \left(\left(\omega_i^\alpha\right)^2 + \frac{\varepsilon \Omega^2}{2\left(1+\varepsilon\cos(\Omega t)\right)^2 }\left(\cos(\Omega t) + \frac{\varepsilon}{4}\big(3 + \cos(2\Omega t)\big)\right) \right) \Psi = 0.\end{equation}
	In general, we cannot obtain the quasifrequencies of this Hill equation in a closed form. Nevertheless, under the assumption that $\varepsilon = o(1)$, to leading order we obtain the Mathieu equation
	\begin{equation}
		\phi''(\tau) + (a-2q\cos(2\tau))\phi(\tau) = 0, \quad \text{where} \quad \tau = \frac{\Omega t}{2}, \quad \phi(\tau) = \Psi(t), \quad a = \left(\frac{2\omega_{\mathrm{s},i}^\alpha}{\Omega}\right)^2, \quad q = -\varepsilon.
	\end{equation}
	Now we consider the limit as $\delta \rightarrow 0$. If $\Omega = O(\sqrt{\delta})$, we have $q/a = o(1)$, as in Case (i) of the previous section. However, if $\Omega$ is fixed and $\varepsilon = O(\delta)$, we have $a,q = O(\delta)$, which corresponds to Case (iii) in the previous section.

		\subsection{Numerical computations} \label{sec:num_unimod}
		Here, and throughout this work, we perform the numerical computations in a two-dimensional structure with circular resonators of radius $R=0.1$ with static material parameters $\rho_0 = \kappa_0 = 1, \rho_\r = \kappa_\r=9000$. In this section we compute the band structure in the case
		\begin{equation}\kappa_t(t) = 1, \qquad \rho_t(t) = \frac{1}{1 + \varepsilon\cos(\Omega t)}, \quad 0 \leq \varepsilon < 1,\end{equation}
		for two different geometries. Here, the static band structure $\omega_{\mathrm{s},i}$ is computed using the multipole method as in \cite[Appendix C]{ammari2017subwavelength}. \hl{The characteristic exponents can be efficiently approximated through truncation of the determinant of an infinite matrix} (see \textit{e.g.} \cite[\S4.23]{mclachlan1951theory},\cite{strang2005characteristic}).
		
		\subsubsection{Square lattice of resonators}\label{sec:SL_uni}
		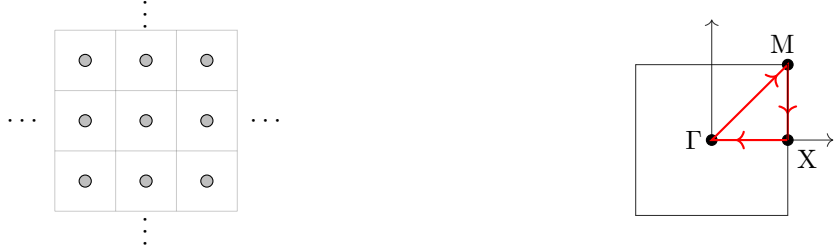
\begin{figure}[p]
			\begin{subfigure}[b]{0.48\linewidth}
				\centering
				\begin{tikzpicture}[scale=0.8]
					\begin{scope}[xshift=-5cm,scale=1]
						\pgfmathsetmacro{\r}{0.1pt} 
						\coordinate (a) at (1,0);		
						\coordinate (b) at (0,1);
						
						\draw[opacity=0.2] (0,0) -- (1,0);
						\draw[opacity=0.2] (0,0) -- (0,1);
						\draw[fill=lightgray] (0.5,0.5) circle(\r); 
						
						\begin{scope}[shift = (a)]					
							\draw[opacity=0.2] (1,0) -- (1,1);
							\draw[opacity=0.2] (0,0) -- (1,0);
							\draw[opacity=0.2] (0,0) -- (0,1);
							\draw[fill=lightgray] (0.5,0.5)  circle(\r); 
						\end{scope}
						\begin{scope}[shift = (b)]
							\draw[opacity=0.2] (1,1) -- (0,1);
							\draw[opacity=0.2] (0,0) -- (1,0);
							\draw[opacity=0.2] (0,0) -- (0,1);
							\draw[fill=lightgray] (0.5,0.5) circle(\r); 
						\end{scope}
						\begin{scope}[shift = ($-1*(a)$)]
							\draw[opacity=0.2] (0,0) -- (1,0);
							\draw[opacity=0.2] (0,0) -- (0,1);
							\draw[fill=lightgray] (0.5,0.5)  circle(\r); 
						\end{scope}
						\begin{scope}[shift = ($-1*(b)$)]
							\draw[opacity=0.2] (0,0) -- (1,0);
							\draw[opacity=0.2] (0,0) -- (0,1);
							\draw[fill=lightgray] (0.5,0.5) circle(\r); 
						\end{scope}
						\begin{scope}[shift = ($(a)+(b)$)]
							\draw[opacity=0.2] (1,0) -- (1,1) -- (0,1);
							\draw[opacity=0.2] (0,0) -- (1,0);
							\draw[opacity=0.2] (0,0) -- (0,1);
							\draw[fill=lightgray] (0.5,0.5) circle(\r);
						\end{scope}
						\begin{scope}[shift = ($-1*(a)-(b)$)]
							\draw[opacity=0.2] (0,0) -- (1,0);
							\draw[opacity=0.2] (0,0) -- (0,1);
							\draw[fill=lightgray] (0.5,0.5) circle(\r);
						\end{scope}
						\begin{scope}[shift = ($(a)-(b)$)]
							\draw[opacity=0.2] (1,0) -- (1,1);
							\draw[opacity=0.2] (0,0) -- (1,0);
							\draw[opacity=0.2] (0,0) -- (0,1);
							\draw[fill=lightgray] (0.5,0.5) circle(\r);
						\end{scope}
						\begin{scope}[shift = ($-1*(a)+(b)$)]
							\draw[opacity=0.2] (1,1) -- (0,1);
							\draw[opacity=0.2] (0,0) -- (1,0);
							\draw[opacity=0.2] (0,0) -- (0,1);
							\draw[fill=lightgray] (0.5,0.5) circle(\r);
						\end{scope}
						\begin{scope}[shift = ($2*(a)$)]
							\draw (0.5,0.5) node[rotate=0]{$\cdots$};
						\end{scope}
						\begin{scope}[shift = ($-2*(a)$)]
							\draw (0.5,0.5) node[rotate=0]{$\cdots$};
						\end{scope}
						\begin{scope}[shift = ($2*(b)$)]
							\draw (0.5,0.3) node[rotate=90]{$\cdots$};
						\end{scope}
						\begin{scope}[shift = ($-2*(b)$)]
							\draw (0.5,0.7) node[rotate=90]{$\cdots$};
						\end{scope}
					\end{scope}
				\end{tikzpicture}
				\caption{Circular resonators in square lattice.}
			\end{subfigure}
			\begin{subfigure}[b]{0.48\linewidth}
				\centering
				\begin{tikzpicture}[scale=2]	
					\coordinate (a) at ({1/sqrt(3)},1);		
					\coordinate (b) at ({1/sqrt(3)},-1);
					\coordinate (c) at ({2/sqrt(3)},0);
					\coordinate (M) at (0.5,0.5);
					\coordinate (M2) at (-0.5,0.5);
					\coordinate (M3) at (-0.5,-0.5);
					\coordinate (M4) at (0.5,-0.5);
					\coordinate (X) at (0.5,0);
					
					\draw[->,opacity=0.8] (0,0) -- (0.8,0);
					\draw[->,opacity=0.8] (0,0) -- (0,0.8);
					
					\draw[fill] (M) circle(1pt) node[yshift=8pt, xshift=-2pt]{M}; 
					\draw[fill] (0,0) circle(1pt) node[left]{$\Gamma$}; 
					\draw[fill] (X) circle(1pt) node[below right]{X}; 
					
					\draw[thick, postaction={decorate}, decoration={markings, mark=at position 0.2 with {\arrow{>}}, markings, mark=at position 0.65 with {\arrow{>}}, markings, mark=at position 0.9 with {\arrow{>}}}, color=red]
					(X) -- (0,0) -- (M) -- (X);
					\draw[opacity=0.8] (M) -- (M2) -- (M3) -- (M4) -- cycle; 
				\end{tikzpicture}
				\vspace{15pt}
				
				\caption{Brillouin zone and the symmetry points $\Gamma$, $\mathrm{X}$ and $\mathrm{M}$.}
			\end{subfigure}
			\caption{Illustration of the square lattice, and corresponding Brillouin zone. The red path shows the points where the band functions are computed in \Cref{fig:SLuni}.} \label{fig:square_uni}
		\end{figure}
		We begin by considering resonators in a square lattice defined through the lattice vectors
		\begin{equation}l_1 = \begin{pmatrix}1\\0\end{pmatrix}, \quad l_2 = \begin{pmatrix}0\\1\end{pmatrix}.\end{equation}
		The lattice and corresponding Brillouin zone is illustrated in \Cref{fig:square_uni}. The symmetry points in $Y^*$ are given by $\Gamma = (0,0), \ \text{M} = (\pi,\pi)$ and $\text{X}=(\pi,0)$.
		
		The static ($\varepsilon = 0$) subwavelength band structure of the system is illustrated in \Cref{fig:SLa}. Although often restricted to positive frequencies, the band structure is symmetric around $\omega = 0$. \Cref{fig:SLb} shows the same band structure, but folded around the frequency $\Omega = 0.2$, which is the same frequency as in \Cref{fig:SLc} where modulation occurs with $\varepsilon = 0.3$. When the modulation is introduced, the frequencies at the edges of the Brillouin zone open into $k$-gaps. \hl{We observe that the group velocity, defined as the slope of the band functions, become infinite at the edges of these $k$-gaps} \cite{martinez2017standing}.
		
		\begin{figure}[p] 
			\begin{subfigure}[b]{0.3\linewidth}
				\begin{center}
					\includegraphics[width=1\linewidth]{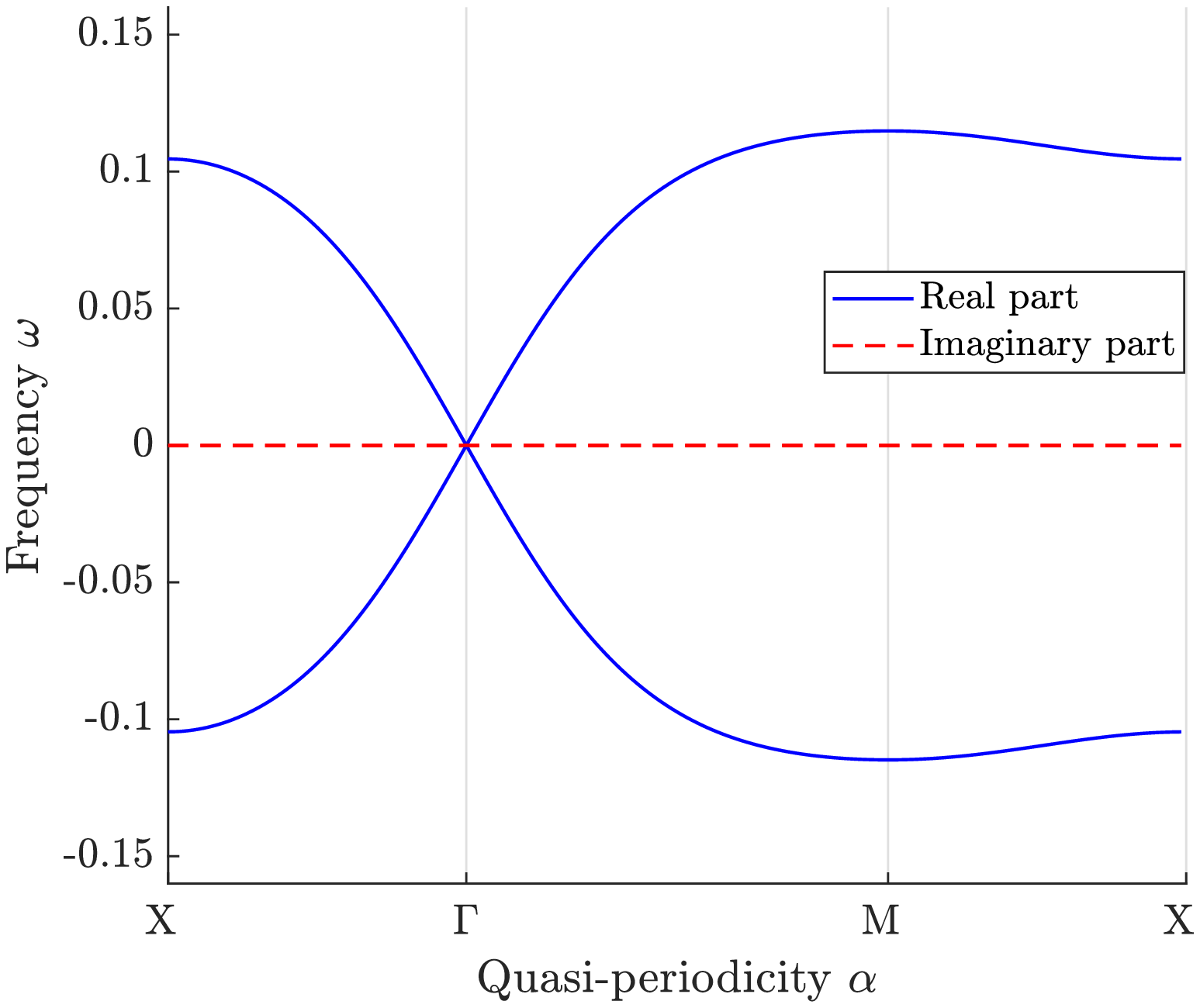}
				\end{center}
				\caption{Subwavelength bands in the unmodulated case.} \label{fig:SLa}
			\end{subfigure}
			\hspace{10pt}
			\begin{subfigure}[b]{0.3\linewidth}
				\begin{center}
					\includegraphics[width=1\linewidth]{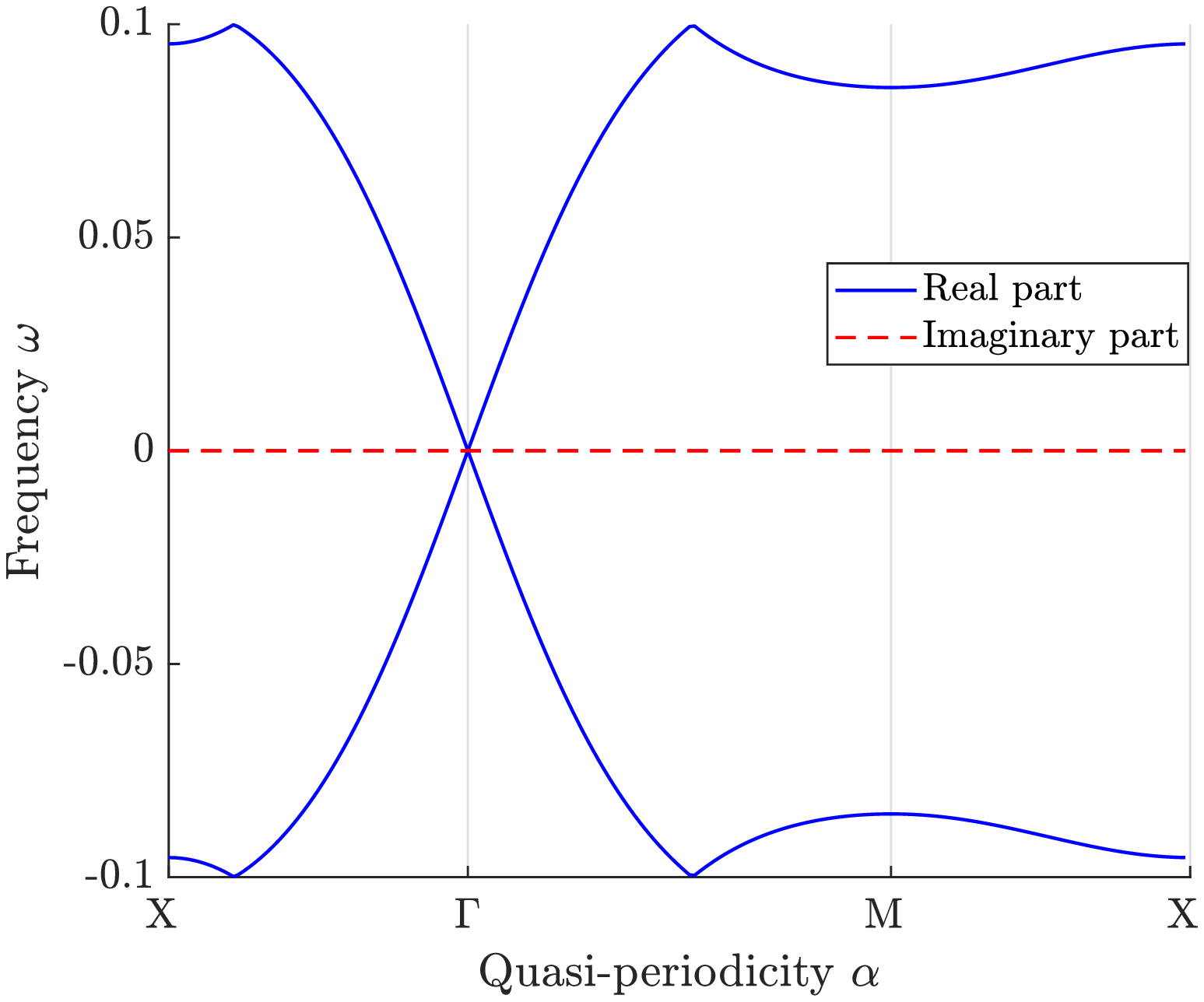}
				\end{center}
				\caption{Same as \Cref{fig:SLa} but folded with $\Omega = 0.2$.}\label{fig:SLb}
			\end{subfigure}
			\hspace{10pt}
			\begin{subfigure}[b]{0.3\linewidth}
				\begin{center}
					\includegraphics[width=1\linewidth]{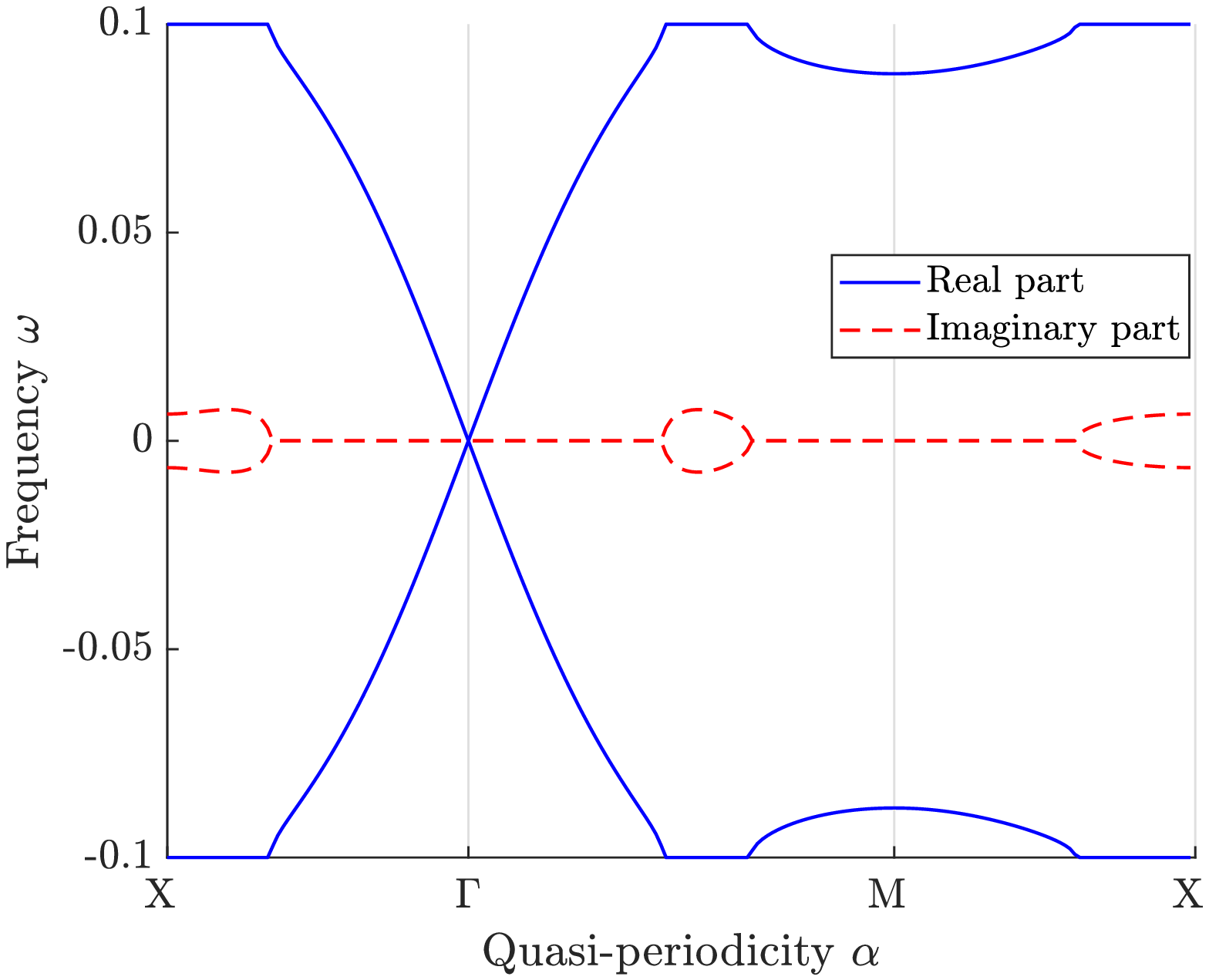}
				\end{center}
				\caption{$\rho$-modulated structure with $\Omega = 0.2$ and $\varepsilon = 0.3$.}\label{fig:SLc}
			\end{subfigure}
			\caption{Band structure of high-contrast resonators in a square lattice. As the modulation increases, $k$-gaps open around the edges of the time-Brillouin zone $[-\Omega/2,\Omega/2)$.}\label{fig:SLuni}
		\end{figure}

\subsubsection{Honeycomb lattice of resonators} \label{sec:HL_uni}
Next, we illustrate the case of a honeycomb lattice of resonators, illustrated in \Cref{fig:honeycomb_uni}. The lattice vectors are given by
\begin{equation}l_1 = \begin{pmatrix}3\\\sqrt{3}\end{pmatrix}, \quad l_2 = \begin{pmatrix}3\\-\sqrt{3}\end{pmatrix}\end{equation}
and the unit cell contains two resonators $D_1$ and $D_2$ centred around $c_1=(1,0)$ and $c_2 = (2,0)$, respectively. The symmetry points in $Y^*$ are given by $\Gamma = (0,0), \ \text{M} = \alpha_1/2$ and $\text{X}=2\alpha_1/3+\alpha_2/3$.

\Cref{fig:HLuni} shows the modulated subwavelength band structure for different values of $\Omega$, when $\varepsilon=0.2$ is fixed. \Cref{fig:HLa} shows the case $\Omega=0.3$. In this case, the static band structure is not folded and the modulated band structure is qualitatively similar to the static case. \Cref{fig:HLb} corresponds to $\Omega=0.23$, where complex frequencies occur. However, some bands remain real and there is no full $k$-gap. \Cref{fig:HLc} shows the band structure for $\Omega=0.2$, where all band functions are complex at $\alpha = \mathrm{K}$ and thus exhibit a $k$-gap.
		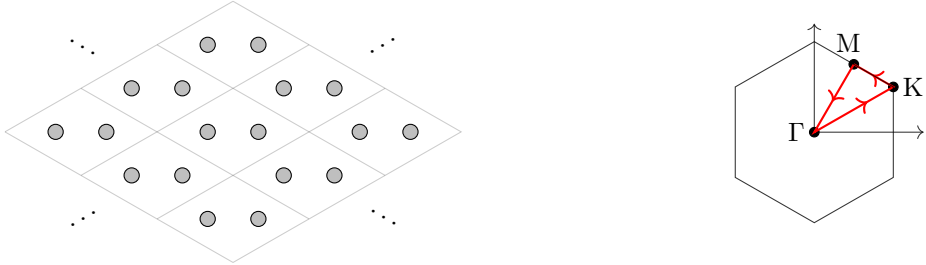
\begin{figure}[p]
	\begin{subfigure}[b]{0.49\linewidth}
		\centering\begin{tikzpicture}[scale=1]
				\pgfmathsetmacro{\r}{0.1pt}
				\coordinate (a) at (1,{1/sqrt(3)});		
				\coordinate (b) at (1,{-1/sqrt(3)});
				
				\draw[opacity=0.2] (0,0) -- (a);
				\draw[opacity=0.2] (0,0) -- (b);
				\draw[fill=lightgray] ({2/3},0) circle(\r);
				\draw[fill=lightgray] ({4/3},0) circle(\r);
				
				\begin{scope}[shift = (a)]					
					\draw[opacity=0.2] (0,0) -- (1,{1/sqrt(3)});
					\draw[opacity=0.2] (0,0) -- (1,{-1/sqrt(3)});
					\draw[opacity=0.2] (1,{1/sqrt(3)}) -- (2,0);
					\draw[fill=lightgray] ({2/3},0) circle(\r);
					\draw[fill=lightgray] ({4/3},0) circle(\r);
				\end{scope}
				\begin{scope}[shift = (b)]
					
					\draw[opacity=0.2] (0,0) -- (1,{1/sqrt(3)});
					\draw[opacity=0.2] (0,0) -- (1,{-1/sqrt(3)});
					\draw[opacity=0.2] (2,0) -- (1,{-1/sqrt(3)});
					\draw[fill=lightgray] ({2/3},0) circle(\r);
					\draw[fill=lightgray] ({4/3},0) circle(\r);
				\end{scope}
				\begin{scope}[shift = ($-1*(a)$)]
					
					\draw[opacity=0.2] (0,0) -- (1,{1/sqrt(3)});
					\draw[opacity=0.2] (0,0) -- (1,{-1/sqrt(3)});
					\draw[fill=lightgray] ({2/3},0) circle(\r);
					\draw[fill=lightgray] ({4/3},0) circle(\r);
				\end{scope}
				\begin{scope}[shift = ($-1*(b)$)]
					
					\draw[opacity=0.2] (0,0) -- (1,{1/sqrt(3)});
					\draw[opacity=0.2] (0,0) -- (1,{-1/sqrt(3)});
					\draw[fill=lightgray] ({2/3},0) circle(\r);
					\draw[fill=lightgray] ({4/3},0) circle(\r);
				\end{scope}
				\begin{scope}[shift = ($(a)+(b)$)]
					
					\draw[opacity=0.2] (0,0) -- (1,{1/sqrt(3)});
					\draw[opacity=0.2] (0,0) -- (1,{-1/sqrt(3)});
					\draw[opacity=0.2] (1,{1/sqrt(3)}) -- (2,0) -- (1,{-1/sqrt(3)});
					\draw[fill=lightgray] ({2/3},0) circle(\r);
					\draw[fill=lightgray] ({4/3},0) circle(\r);
				\end{scope}
				\begin{scope}[shift = ($-1*(a)-(b)$)]
					
					\draw[opacity=0.2] (0,0) -- (1,{1/sqrt(3)});
					\draw[opacity=0.2] (0,0) -- (1,{-1/sqrt(3)});
					\draw[fill=lightgray] ({2/3},0) circle(\r);
					\draw[fill=lightgray] ({4/3},0) circle(\r);
				\end{scope}
				\begin{scope}[shift = ($(a)-(b)$)]
					
					\draw[opacity=0.2] (0,0) -- (1,{1/sqrt(3)});
					\draw[opacity=0.2] (0,0) -- (1,{-1/sqrt(3)});
					\draw[opacity=0.2] (1,{1/sqrt(3)}) -- (2,0);
					\draw[fill=lightgray] ({2/3},0) circle(\r);
					\draw[fill=lightgray] ({4/3},0) circle(\r);
				\end{scope}
				\begin{scope}[shift = ($-1*(a)+(b)$)]					
					\draw[opacity=0.2] (0,0) -- (1,{1/sqrt(3)});
					\draw[opacity=0.2] (0,0) -- (1,{-1/sqrt(3)});
					\draw[opacity=0.2] (2,0) -- (1,{-1/sqrt(3)});
					\draw[fill=lightgray] ({2/3},0) circle(\r);
					\draw[fill=lightgray] ({4/3},0) circle(\r);
				\end{scope}
				\begin{scope}[shift = ($2*(a)$)]
					\draw (1,0) node[rotate=30]{$\cdots$};
				\end{scope}
				\begin{scope}[shift = ($-2*(a)$)]
					\draw (1,0) node[rotate=210]{$\cdots$};
				\end{scope}
				\begin{scope}[shift = ($2*(b)$)]
					\draw (1,0) node[rotate=-30]{$\cdots$};
				\end{scope}
				\begin{scope}[shift = ($-2*(b)$)]
					\draw (1,0) node[rotate=150]{$\cdots$};
				\end{scope}
		\end{tikzpicture}
		\caption{Circular resonators in a hexagonal lattice.}
	\end{subfigure}
	\begin{subfigure}[b]{0.49\linewidth}
		\centering
			\begin{tikzpicture}[scale=1.8]	
				\coordinate (a) at ({1/sqrt(3)},1);		
				\coordinate (b) at ({1/sqrt(3)},-1);
				\coordinate (c) at ({2/sqrt(3)},0);
				\coordinate (M) at ({0.5/sqrt(3)},0.5);
				\coordinate (K1) at ({1/sqrt(3)},{1/3});
				\coordinate (K2) at ({1/sqrt(3)},{-1/3});
				\coordinate (K3) at (0,{-2/3});
				\coordinate (K4) at ({-1/sqrt(3)},{-1/3});
				\coordinate (K5) at ({-1/sqrt(3)},{1/3});
				\coordinate (K6) at (0,{2/3});
				
				\draw[->,opacity=0.8] (0,0) -- (0.8,0);
				\draw[->,opacity=0.8] (0,0) -- (0,0.8);
				\draw[fill] (M) circle(1pt) node[yshift=8pt, xshift=-2pt]{M}; 
				\draw[fill] (0,0) circle(1pt) node[left]{$\Gamma$}; 
				\draw[fill] (K1) circle(1pt) node[right]{K}; 
				
				\draw[thick,postaction={decorate}, decoration={markings, mark=at position 0.2 with {\arrow{>}}, markings, mark=at position 0.65 with {\arrow{>}}, markings, mark=at position 0.9 with {\arrow{>}}}, color=red]
				(M) -- (0,0) -- (K1) -- (M);
				\draw[opacity=0.8] (K1) -- (K2) -- (K3) -- (K4) -- (K5) -- (K6) -- cycle; 
			\end{tikzpicture}
		\vspace{15pt}		
		\caption{Brillouin zone and the symmetry points $\Gamma$, $\mathrm{K}$ and $\mathrm{M}$.}
	\end{subfigure}
	\caption{Illustration of the honeycomb lattice of resonators, and corresponding Brillouin zone. The red path shows the points where the band functions are computed in \Cref{fig:HLuni}.} \label{fig:honeycomb_uni}
\end{figure}

\begin{figure}[p] 
	\begin{subfigure}[b]{0.3\linewidth}
		\begin{center}
			\includegraphics[width=1\linewidth]{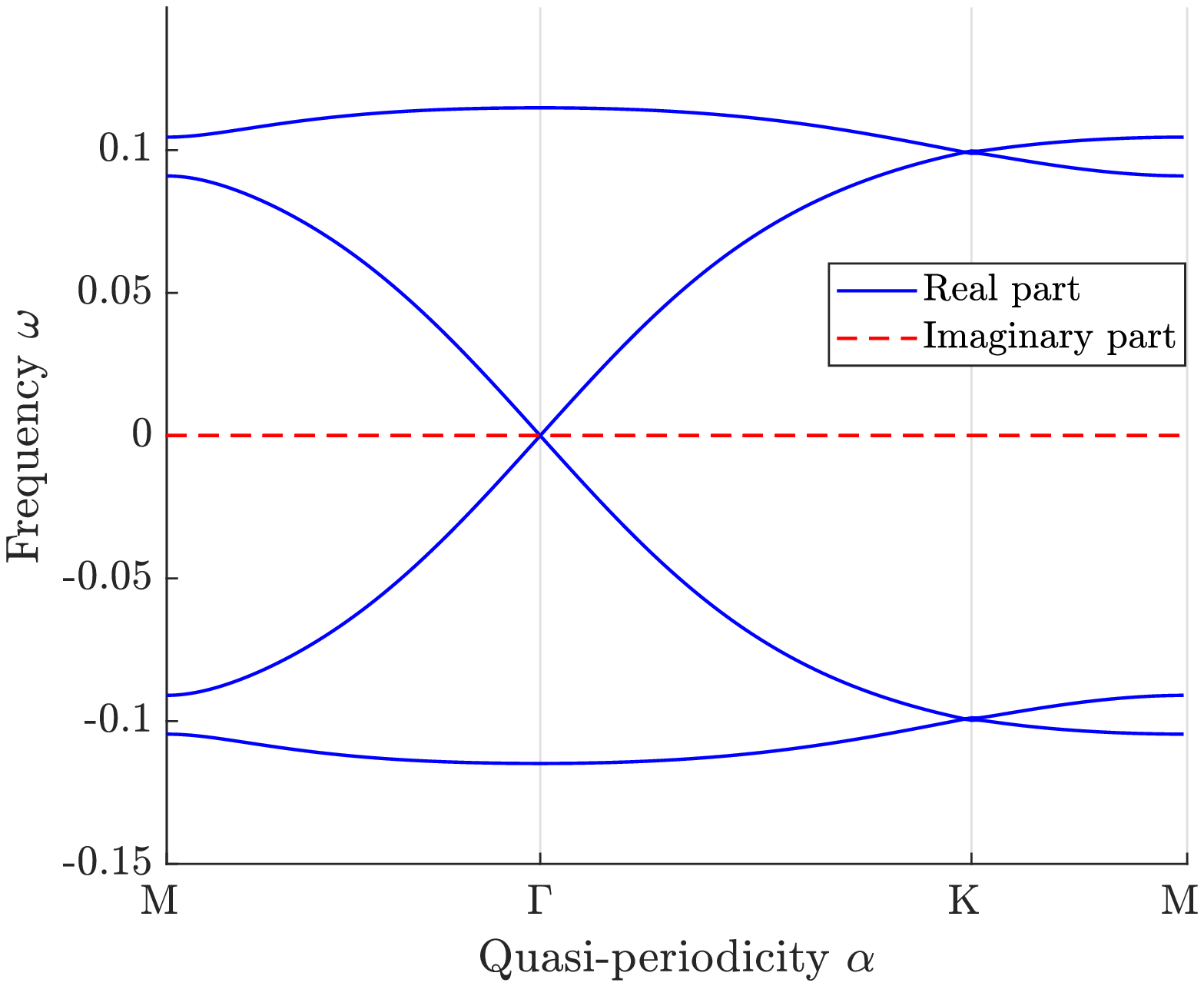}
		\end{center}
		\caption{$\varepsilon = 0.2$ and $\Omega = 0.3$, showing largely unaltered band structure compared to static case.} \label{fig:HLa}
	\end{subfigure}
	\hspace{10pt}
	\begin{subfigure}[b]{0.3\linewidth}
		\begin{center}
			\includegraphics[width=1\linewidth]{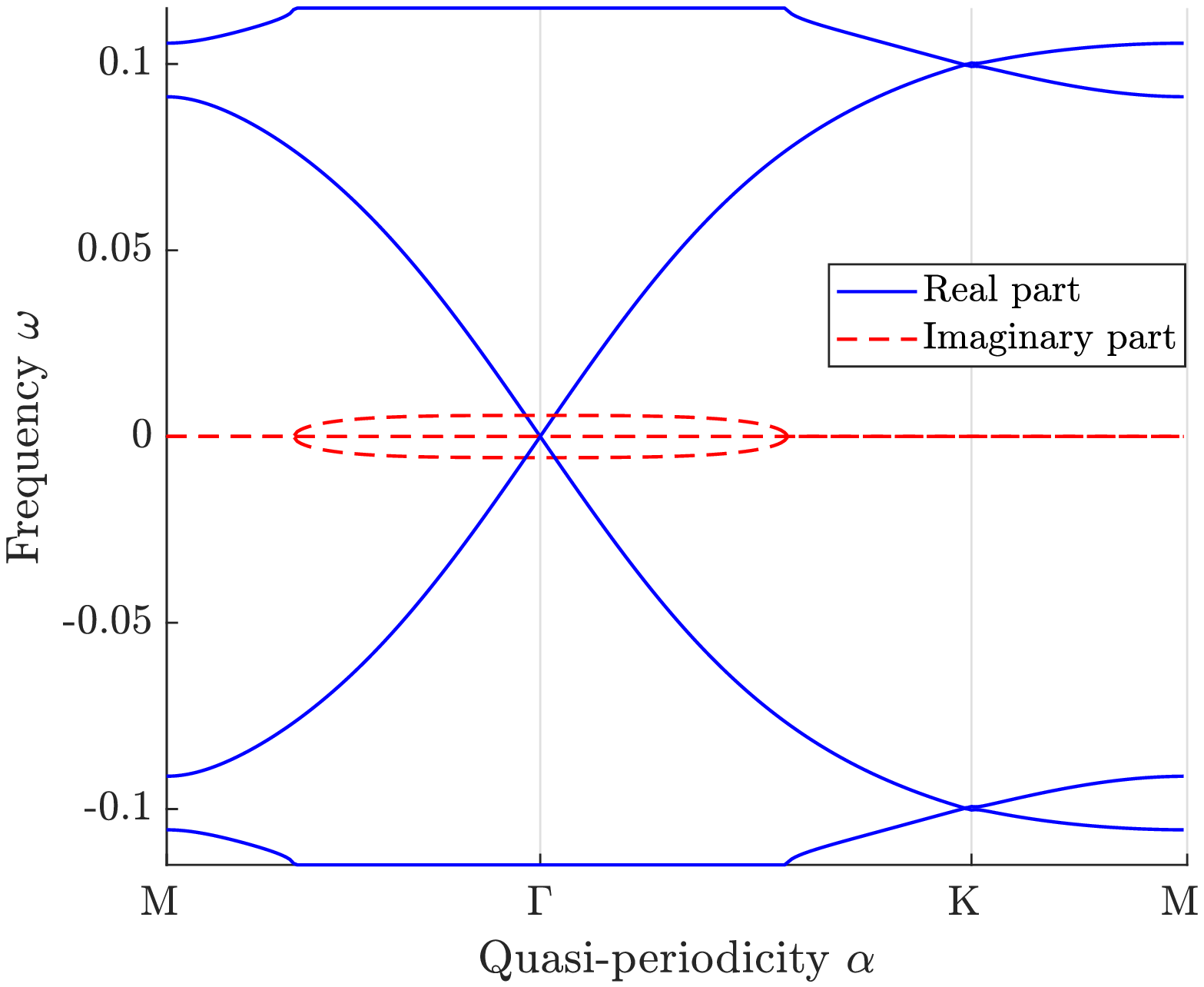}
		\end{center}
		\caption{$\varepsilon = 0.2$ and $\Omega = 0.23$, showing complex band frequencies but no full $k$-gap.}\label{fig:HLb}
	\end{subfigure}
	\hspace{10pt}
	\begin{subfigure}[b]{0.3\linewidth}
		\begin{center}
			\includegraphics[width=1\linewidth]{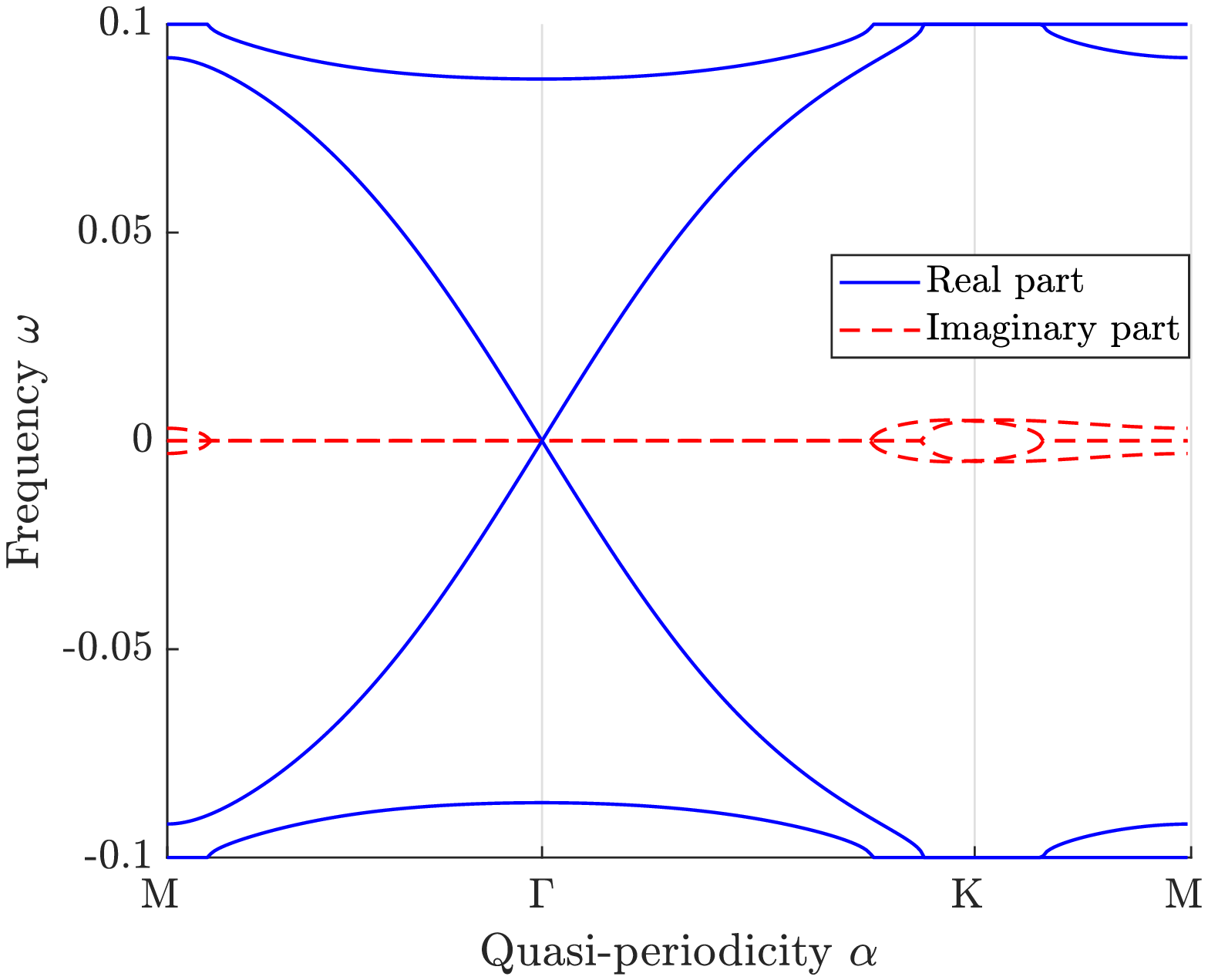}
		\end{center}
		\caption{$\varepsilon = 0.2$ and $\Omega = 0.2$, showing a $k$-gap around the symmetry point $\mathrm{K}$ in the Brillouin zone.}\label{fig:HLc}
	\end{subfigure}
	\caption{Band structure of high-contrast resonators in a honeycomb lattice. For a fixed modulation strength $\varepsilon$, complex band frequencies may appear for certain $\Omega$. A $k$-gap opens when the degenerate static frequencies become complex (\Cref{fig:HLc}).} \label{fig:HLuni}
\end{figure}
				

\section{Resonator-modulated systems} \label{sec:resonatormod}
In this section, we study the case when the time-modulation is only applied to the interior of the resonators, while the surrounding material is constant in $t$. Due to the time-modulation, we will obtain a system of coupled Helmholtz equations in the frequency domain, posed at frequencies which differ by multiples of $\Omega$.

We consider the periodic structure $\C$ as defined in \Cref{sec:formulation} and now assume
\begin{equation} \label{eq:resonatormod}
\kappa(x,t) = \begin{cases}
	 \kappa_0, & x \in \R^d \setminus \overline{\C}, \\  \kappa_\r\kappa_i(t), & x\in \C_i,
\end{cases}, \qquad \rho(x,t) = \begin{cases}
	\rho_0, & x \in \R^d \setminus \overline{\C}, \\  \rho_\r\rho_i(t), & x\in \C_i. \end{cases}
\end{equation}
The functions $\rho_i(t)$ and $\kappa_i(t)$ describe the modulation inside the $i$\textsuperscript{th} resonator $D_i$, and we assume that each $\rho_i, \kappa_i$ is periodic with period $\Omega$. Moreover, we assume that $\kappa_i \in C^1(\R)$ for each $i=1,...,N$. 

In \Cref{sec:cap}, we prove the capacitance matrix approximation of the subwavelength band structure in the limit as $\delta\to 0$. \hl{As in }\Cref{sec:cap_static}, \hl{the capacitance formulation provides a way to integrate the spatial part of the equation. In this time-dependent setting, the resulting equation is an ordinary differential equation posed in $t$.} As we shall see, the dynamics is now governed by a system of coupled Hill differential equations.

\subsection{Capacitance matrix formulation of the problem} \label{sec:cap}
In this section, we will use the above notion of capacitance in order to derive an asymptotic expansion of the subwavelength band functions as $\delta\to0$. Throughout this section we assume that $\alpha \neq 0$.

We seek solutions to \eqref{eq:wave_transf} under the modulation specified in \eqref{eq:resonatormod}. From the regularity of $\rho$ and $\kappa$ we know that $u$ is continuously differentiable in $t$ \cite{lions1988controlabilite}. Since $e^{- \iu \omega t}u(x,t) $ is a $T$-periodic function of $t$ we have a Fourier series expansion as 
\begin{equation}u(x,t)= e^{\iu \omega t}\sum_{n = -\infty}^\infty v_n(x)e^{\iu n\Omega t}.\end{equation}
In the time domain, we have the transmission conditions at $x\in \p D_i$
\begin{equation}\delta \frac{\partial {u}}{\partial \nu} \bigg|_{+} - \frac{1}{\rho_i(t)}\frac{\partial {u}}{\partial \nu} \bigg|_{-} = 0, \qquad x\in \p D_i, \ t\in \R.\end{equation}
In the frequency domain, we then have the following equation, for $n\in \Z$:
\begin{equation} \label{eq:freq}
	\left\{
	\begin{array} {ll}
		\ds \Delta {v_n}+ \frac{\rho_0(\omega+n\Omega)^2}{\kappa_0} {v_n}  = 0 & \text{in } Y \setminus \overline{D}, \\[0.3em]
		\ds \Delta v_{i,n}^* +\frac{\rho_\r(\omega+n\Omega)^2}{\kappa_\r} v_{i,n}^{**}  = 0 & \text{in } D_i, \\
		\nm
		\ds  {v_n}|_{+} -{v_n}|_{-}  = 0  & \text{on } \partial D, \\
		\nm
		\ds  \delta \frac{\partial {v_n}}{\partial \nu} \bigg|_{+} - \frac{\partial v_{i,n}^* }{\partial \nu} \bigg|_{-} = 0 & \text{on } \partial D_i, \\[0.3em]
		v_n(x)e^{\iu \alpha\cdot x} \text{ is $\Lambda$-periodic in $x$}.
	\end{array}
	\right.
\end{equation}
Here, $v_{i,n}^*(x)$ and $v_{i,n}^{**}(x)$ are defined through the convolutions
\begin{equation}v_{i,n}^*(x) = \sum_{m = -\infty}^\infty r_{i,m} v_{n-m}(x), \quad  v_{i,n}^{**}(x) = \frac{1}{\omega+n\Omega}\sum_{m = -\infty}^\infty k_{i,m}\big(\omega+(n-m)\Omega\big)v_{n-m}(x),\end{equation}
where $r_{i,m}$ and $k_{i,m}$ are the Fourier series coefficients of $1/\rho_i$ and $1/\kappa_i$, respectively:
\begin{equation}\frac{1}{\rho_i(t)} = \sum_{n = -\infty}^\infty r_{i,n} e^{\iu n \Omega t}, \quad \frac{1}{\kappa_i(t)} = \sum_{n = -\infty}^\infty k_{i,n} e^{\iu n \Omega t}.\end{equation}
Observe that \eqref{eq:freq} consists of coupled Helmholtz equations at frequencies differing by integer multiples of $\Omega$. The coupling of the Helmholtz equations is specified through $\rho_i$ and $\kappa_i$. Rescaling a solution to \eqref{eq:freq} produces another solution, so we assume that the solution is normalized as $\|v_0\|_{H^1(Y)} = 1$. Since $u$ is continuously differentiable in $t$, we then have as $n\to \infty$,
\begin{equation} \label{eq:reg_v}
	\|v_n\|_{H^1(Y)} = o\left(\frac{1}{n}\right).
\end{equation}  

We will consider the case when the modulation of $\rho$ \hl{and $\kappa$} consist of finite Fourier series with a large number of nonzero Fourier coefficients: 
\begin{equation}\label{eq:finitefourier}
	\frac{1}{\rho_i(t)} = \sum_{n = -M}^M r_{i,n} e^{\iu n \Omega t}, \qquad \mathhl{\frac{1}{\kappa_i(t)} = \sum_{n = -M}^M k_{i,n} e^{\iu n \Omega t}}\end{equation}
for some $M\in \N$ satisfying
\begin{equation}M = O\left(\delta^{-\gamma/2}\right),\end{equation}
for some $0 < \gamma < 1$. We seek subwavelength quasifrequencies $\omega$ of the wave equation \eqref{eq:wave_transf} in the sense of \Cref{def:sub}. In particular, we assume that $\omega$, and also the frequency $\Omega$ of the modulation, is of the same order as the static subwavelength resonant frequencies:
\begin{equation}\omega = O\left(\delta^{1/2}\right), \qquad \Omega = O\left(\delta^{1/2}\right).\end{equation}
We then have the following result.
\begin{lem}\label{lem:const}
	As $\delta \to 0$, the functions $v^*_{i,n}(x)$ are approximately constant for $x$ inside $D_i$, \ie{},
	\begin{align*}
		v^*_{i,n}(x) &= c_{i,n} + O(\delta^{(1-\gamma)/2}), \quad x \in D_i,
	\end{align*}
uniformly for $n\leq M$, for some constants $c_{i,n}$, $i=1,...,N$.
\end{lem}
\begin{proof}
	Since $v^*_{i,n}$ \hl{and $v^{**}_{i,n}$ } are defined through finite convolutions, we have from \eqref{eq:reg_v} that
	\begin{equation}
		\|v_{i,n}^*\|_{H^1(Y)} \leq \frac{K^*}{n}, \qquad 	\|v_{i,n}^{**}\|_{H^1(Y)} \leq \frac{K^{**}}{n}, \quad n\neq 0,  \quad
	\end{equation}
	\hl{for some constants $K^*, K^{**}$ satisfying $K^*,K^{**} = O(1)$ as $\delta \to 0$.} We then have from \eqref{eq:freq} that
	\begin{equation}\int_{\p D_i} v_{i,n}^* \frac{\p v_{i,n}^*}{\p \nu} \dx \sigma = O(\delta), \quad \int_{D_i} v_{i,n}^*\Delta v_{i,n}^* \dx x = O(\delta^{1-\gamma}).\end{equation}
	Using integration by parts, we obtain
	\begin{align}
		\int_{D_i} |\nabla v_{i,n}^*|^2 \dx x &= \int_{\p D_i} v_{i,n}^* \frac{\p v_{i,n}^*}{\p \nu} \dx \sigma - \int_{D_i} v_{i,n}^*\Delta v_{i,n}^* \dx x \\
		&= O(\delta^{1-\gamma}).
	\end{align}
	Therefore, for $|n| \leq M$  we have $v_{i,n}^*(x) = c_{i,n} + O(\delta^{(1-\gamma)/2})$ for $x\in D_i$, which proves the claim.
\end{proof}
In $Y\setminus \overline{D}$, the solution $v$ satisfies the Helmholtz equation and can be represented through the single layer potential as
\begin{equation}v_n(x) = \mathcal{S}_D^{\alpha,\omega+n\Omega}[\phi_n](x),\end{equation}
for some $\phi_n$, $n\in \Z$. From \eqref{eq:reg_v} we have that $\|\phi_n\|_{L^2(\p D)} \leq \widetilde K/n$ as $n\to \infty$, for some constant $\widetilde K$. Using \eqref{eq:jump1} and \eqref{eq:Sexp} we have on the boundary $\p D_i$: 
\begin{align}
	v_{i,n}^*(x) &= \sum_{m = -M}^M r_{i,m}\mathcal{S}_D^{\alpha,\omega+(n-m)\Omega}[\phi_{n-m}](x) \\
	&= \mathcal{S}_D^{\alpha,0}\left[\sum_{m = -M}^M r_{i,m}\phi_{n-m}\right](x) + O(\delta^{1-\gamma}),
\end{align}
as $\delta \rightarrow 0$, where the error term is uniform for $|n|\leq M$. From \Cref{lem:const}  we then find that 
\begin{equation} \label{eq:conv+basis}
	\sum_{m = -M}^M r_{i,m}\phi_{n-m} = \sum_{i=1}^N c_{i,n} \psi_i^\alpha + O(\delta^{(1-\gamma)/2}),
\end{equation}
for $\alpha \neq 0$, where $\psi_i^\alpha$ are the basis functions defined in \eqref{eq:psiC}. 

On one hand, using the transmission conditions and integration by parts, we obtain
\begin{equation}\label{eq:first}
\int_{\p D_i} \frac{\p v_n}{\p \nu}\bigg|_+ \dx\sigma = \frac{1}{\delta} \int_{\p D_i} \frac{\p v^*_{i,n}(x)}{\p \nu}\bigg|_- \dx \sigma(x) = -\frac{\rho_\r(\omega+n\Omega)^2}{\delta\kappa_\r}V_{i,n}^{**},
\end{equation}
where $V_{i,n}^{**}$ is defined as 
\begin{equation}V_{i,n}^{**} = \frac{1}{\omega+n\Omega}\sum_{m = -M}^M k_{i,m}\big(\omega+(n-m)\Omega\big)V_{i,n-m}, \quad V_{i,n} = \int_{D_i} v_n(x)\dx x.\end{equation}
On the other hand, using the jump relation \eqref{eq:jump2}, the asymptotic expansion \eqref{eq:expK} and the integration formula \eqref{eq:intK} yield
\begin{equation}
	\int_{\p D_i} \frac{\p v_n}{\p \nu}\bigg|_+ \dx \sigma =  \int_{\p D_i} \left(\frac{1}{2}I + (\mathcal{K}_D^{-\alpha,0})^*\right)[\phi_n] \dx \sigma + O(\delta^{(1-\gamma)/2}) = \int_{\p D_i}\phi_n \dx \sigma + O(\delta^{(1-\gamma)/2}),
\end{equation}
for $|n| \leq M$. Taking the convolution and using \eqref{eq:conv+basis}, we have
\begin{equation}\label{eq:second}
	\sum_{m=-M}^{M}r_{i,m}\int_{\p D_i} \frac{\p v_{n-m}}{\p \nu}\bigg|_+ \dx \sigma = -\sum_{j=1}^N c_{j,n}C_{ij}^\alpha + O(\delta^{(1-\gamma)/2}) ,
\end{equation}
where $C_{ij}^\alpha$ are the capacitance coefficients defined in \eqref{eq:psiC}. Combining \eqref{eq:first} and \eqref{eq:second} we therefore obtain 
\begin{align}
	\sum_{j=1}^N c_{j,n}C_{ij}^\alpha &= \frac{\rho_\r}{\delta\kappa_\r}\sum_{m=-M}^Mr_{i,m}(\omega+(n-m)\Omega)^2V_{i,n-m}^{**} + O(\delta^{(1-\gamma)/2}). \label{eq:Ckr}
\end{align}
Next we will take the inverse transform of \eqref{eq:Ckr}. Denoting 
\begin{equation}c_i(t) = e^{\iu \omega t}\sum_{n = -\infty}^\infty c_{i,n}e^{\iu n\Omega t}, \quad V_i(t) = e^{\iu \omega t}\sum_{n = -\infty}^\infty V_{i,n}e^{\iu n\Omega t},\end{equation}
we have
\begin{equation}c_i(t) = \frac{V_i(t)}{|D_i|\rho_i(t)},\end{equation}
where $|D_i|$ denotes the volume (or area) of the $i$\textsuperscript{th} resonator $D_i$. Assuming that $v$ corresponds to a subwavelength solution, we have as $\delta \to 0$,
\begin{equation}c_i(t) = e^{\iu \omega t}\sum_{n = -M}^M c_{i,n}e^{\iu n\Omega t} + o(1), \quad V_i(t) = e^{\iu \omega t}\sum_{n = -M}^M V_{i,n}e^{\iu n\Omega t} + o(1).\end{equation}
This together with \eqref{eq:Ckr} proves the following main result.
\begin{thm}\label{thm:main}
	Assume that the material parameters are given by \eqref{eq:resonatormod} and that $\alpha \neq 0$. Then, as $\delta \to 0$, the quasifrequencies $\omega = \omega(\alpha) \in Y^*_t$ to the wave equation \eqref{eq:wave} in the subwavelength regime are, to leading order, given by the quasifrequencies of the system of ordinary differential equations
	\begin{equation}\label{eq:C_ODE}
			\sum_{j=1}^N C_{ij}^\alpha c_j(t) = -\frac{|D_i|\rho_\r}{\delta\kappa_\r}\frac{1}{\rho_i(t)}\frac{\dx}{\dx t}\left(\frac{1}{\kappa_i(t)}\frac{\dx \rho_ic_i}{\dx t}\right),
	\end{equation}
	for $i=1,...,N$.
\end{thm}
We remark that the left-hand side of \eqref{eq:C_ODE} is specified by the entries of the matrix-vector product $C^\alpha c(t)$. In fact, we can rewrite \eqref{eq:C_ODE} into the following system of Hill equations
\begin{equation}	\label{eq:hill}
\Psi''(t)+ M(t)\Psi(t)=0,
\end{equation}
where $\Psi$ is the vector defined as 
\begin{equation}\Psi = \left(\frac{\rho_i(t)}{\sqrt{\kappa_i(t)}}c_i(t)\right)_{i=1}^N\end{equation}
and $M$ is the matrix defined as 
\begin{equation}\label{eq:M}
	M(t) = \frac{\delta \kappa_\r}{\rho_\r}W_1(t)C^\alpha W_2(t) + W_3(t),\end{equation}
with $W_1, W_2$ and $W_3$ being the diagonal matrices with diagonal entries
\begin{equation}\left(W_1\right)_{ii} = \frac{\sqrt{\kappa_i}\rho_i}{|D_i|}, \qquad \left(W_2\right)_{ii} =\frac{\sqrt{\kappa_i}}{\rho_i}, \qquad \left(W_3\right)_{ii} = \frac{\sqrt{\kappa_i}}{2}\frac{\dx }{\dx t}\frac{\kappa_i'}{\kappa_i^{3/2}},\end{equation}
for $i=1,...,N$.

\begin{rmk}
If we assume that all resonators have equal volume and that
\begin{equation}\kappa_i(t)=1, \quad \text{and} \quad \rho_i(t)=\rho_1(t), \ t \in \R, i=1,..,N,\end{equation} 
(in other words that $\kappa$ is unmodulated and all $\rho_i$ coincide), equation \eqref{eq:hill} and the corresponding quasifrequencies $\omega^\alpha_i$ read
\begin{equation}
\Psi''(t) + \frac{\delta\kappa_\r}{|D_1|\rho_\r}C^\alpha \Psi = 0, \qquad \omega_i^\alpha = \sqrt{\frac{\delta\lambda_i^\alpha}{|D_1|}}v_\r + o(1),
\end{equation}
where $\lambda_i^\alpha$ are the eigenvalues of $C^\alpha$. This agrees with the formula for the static subwavelength resonant frequencies given in \Cref{thm:static}, and we remark that this holds even if the system has a nontrivial modulation of  $\rho$ specified by $\rho_1$.
\end{rmk}
\begin{rmk}
	\hl{The above analysis focuses on the case when $\rho$ and $\kappa$ are smooth in $t$ (due to the assumption in }\eqref{eq:finitefourier}). \hl{In the case of parameters which are discontinuous in $t$, the waves will be subject to time-reflection and we expect the physical phenomena, and the mathematical analysis required to study them, to be quite different.}
\end{rmk}

\begin{rmk}
	\hl{Contrary to the case mentioned in} \Cref{rmk:impedance}, \hl{$k$-gaps may open in the resonator-modulated setting even if the impedance is constant in $t$. For the simplest example, consider $N=1$, where the capacitance matrix is just a single number $C_{11}^\alpha =: \mathrm{Cap}_{D,\alpha}$. Choosing $\kappa_1(t) = 1/\rho_1(t) = 1 + \varepsilon\cos(\Omega t)$ renders} \eqref{eq:hill} \hl{fully analogous to the Mathieau equation studied in} \Cref{sec:rhomod} \hl{and illustrated in} \Cref{fig:SLuni}.
\end{rmk}
\subsection{Numerical computations}
\hl{With }\Cref{thm:main} \hl{as starting point, we can now compute the subwavelength band structure very efficiently. Conceptually, we have reduced the four-dimensional partial differential equation} \eqref{eq:wave_transf} \hl{into the ordinary differential equation }\eqref{eq:hill}, \hl{allowing numerical integration using simple techniques.}

We will numerically compute the Floquet exponents of the Hill system of equations  \cite{denk1995floquet,carlson2000eigenvalue} 
\begin{equation} \label{eq:hillagain}
\Psi''(t) + M(t)\Psi(t) = 0.
\end{equation}
This is an $N\times N$ system of ordinary differential equations of second order, which admits a fundamental basis of solutions $\{\psi_{\I,j}(t),\psi_{\II,j}(t)\}_{j=1}^N$ defined through the initial conditions
\begin{equation} \label{eq:ivp}
\psi_{\I,j}^{(i)}(0)  = \delta_{ij}, \quad \left(\psi_{\I,j}^{(i)}\right)'(0) = 0, \qquad \psi_{\II,j}^{(i)}(0)  = 0, \quad \left(\psi_{\II,j}^{(i)}\right)'(0) = \delta_{ij}.
\end{equation}
Here, and throughout this section, bracketed superscripts denote corresponding vector components. We seek quasiperiodic solutions $\psi$ satisfying 
\begin{equation}\psi(t+T) = e^{\iu \omega T} \psi(t).\end{equation}
It is easy to show that this occurs precisely when $e^{\iu \omega T}$ is an eigenvalue of the fundamental solution at $t=T$, which is the $2N\times 2N$-matrix 
\begin{equation}\W = \begin{pmatrix}
	\left(\psi_{\I,j}^{(i)}(T)\right)_{i,j=1}^N & \left(\psi_{\II,j}^{(i)}(T)\right)_{i,j=1}^N \\[0.8em]
	\left(\left(\psi_{\I,j}^{(i)}\right)'(T)\right)_{i,j=1}^N  & 
	\left(\left(\psi_{\II,j}^{(i)}\right)'(T)\right)_{i,j=1}^N 
\end{pmatrix}.\end{equation}
This offers a straightforward numerical algorithm to compute the Floquet exponents $\omega$: we numerically integrate \eqref{eq:hillagain} with initial conditions \eqref{eq:ivp}, and then approximate $\omega$ through the eigenvalues of $\W$.

\hl{As in} \Cref{sec:num_unimod}, \hl{we use the multipole discretization method to compute the capacitance coefficients} \cite[Appendix C]{ammari2017subwavelength}. \hl{Then, for given choices of $\kappa_i$ and $\rho_i$, we can integrate} \eqref{eq:hillagain} \hl{with $M$ given as in} \eqref{eq:M}. \hl{In these simple, proof-of-concept computations, we use standard} \verb+MATLAB+ \hl{routines for the numerical integration.}

\subsubsection{Exceptional point degeneracy in square lattice of dimers}
We begin by considering the band structure of a structure with the same square lattice as in \Cref{sec:SL_uni}, but where the unit cell now contains two resonators $D_1, D_2$ centred at $c_1 = (0.5-1.2R,0.5)$, $c_2=(0.5+1.2R,0.5)$, respectively. The geometry is illustrated in \Cref{fig:square}. We consider the modulation specified by
\begin{equation}\rho_1(t) = 1, \quad \rho_2(t) = 1, \quad \kappa_1(t) = \frac{1}{1+\varepsilon\cos(\Omega t)}, \quad \kappa_2(t) = \frac{1}{1+\varepsilon\cos(\Omega t+\pi)}, \qquad t\in \R,\end{equation}
for $0\leq \varepsilon < 1$. As we will see, this structure may support exceptional points, which (as mentioned in the introduction) are parameter points where the eigenmodes of the system coalesce. 

The band structure of the material is given in \Cref{fig:SLres}. \Cref{fig:SLresa} shows the static band structure (corresponding to $\varepsilon = 0$) folded with $\Omega = 0.26$. This frequency $\Omega$ lies inside a band gap, and there are intersections between the first (unfolded) and the second (folded) bands. \Cref{fig:SLresb} shows the modulated band structure at $\varepsilon = 0.2$, also with $\Omega = 0.26$. In the modulated structure, the intersection points mark transitions from a real to a conjugate-symmetric spectrum. \hl{An exceptional point is a point where $\W$ is deficient: it does not have a basis of eigenvectors, \textit{i.e.} the matrix of eigenvectors is singular. Numerically, we measure this through the condition number, and }\Cref{fig:SLresc} \hl{demonstrates this singularity at the degeneracies found in }\Cref{fig:SLresb}. \hl{Therefore, we conclude that the degeneracies correspond to exceptional points.}

\begin{figure}[tbh]
	\begin{subfigure}[b]{0.48\linewidth}
		\centering
		\begin{tikzpicture}[scale=1.3]
			\pgfmathsetmacro{\r}{0.06pt}
			\pgfmathsetmacro{\rt}{0.12pt}
			\coordinate (a) at (1,0);		
			\coordinate (b) at (0,1);	
			\coordinate (c) at (1,1);
			\coordinate (x1) at ({2/3},0);
			\coordinate (x2) at ({4/3},0);
			\draw[->] (0,0) -- (a) node[xshift=1,yshift=-7]{ $l_1$};
			\draw[->] (0,0) -- (b) node[xshift=-7,yshift=0]{ $l_2$};
			\draw[opacity=0.5] (a) -- (c) -- (b);
			\draw (1,1) node[below right]{$Y$};
			
			\draw[fill=lightgray] (0.5,0.5){} +(0:\rt) circle(\r);
			\draw[fill=lightgray] (0.5,0.5){} +(180:\rt) circle(\r);
		\end{tikzpicture}
		\vspace{0.65cm}
		\caption{Square unit cell $Y$ containing $2$ resonators.}
	\end{subfigure}
	\begin{subfigure}[b]{0.48\linewidth}
		\centering
		\begin{tikzpicture}[scale=0.8]
			\begin{scope}[xshift=-5cm,scale=1]
				\pgfmathsetmacro{\r}{0.06pt}
				\pgfmathsetmacro{\rt}{0.12pt}
				\coordinate (a) at (1,0);		
				\coordinate (b) at (0,1);
				
				\draw[opacity=0.2] (0,0) -- (1,0);
				\draw[opacity=0.2] (0,0) -- (0,1);
				\draw[fill=lightgray] (0.5,0.5){} +(0:\rt) circle(\r);
				\draw[fill=lightgray] (0.5,0.5){} +(180:\rt) circle(\r);
				
				\begin{scope}[shift = (a)]					
					\draw[opacity=0.2] (1,0) -- (1,1);
					\draw[opacity=0.2] (0,0) -- (1,0);
					\draw[opacity=0.2] (0,0) -- (0,1);
					\draw[fill=lightgray] (0.5,0.5){} +(0:\rt) circle(\r);
					\draw[fill=lightgray] (0.5,0.5){} +(180:\rt) circle(\r);
				\end{scope}
				\begin{scope}[shift = (b)]
					\draw[opacity=0.2] (1,1) -- (0,1);
					\draw[opacity=0.2] (0,0) -- (1,0);
					\draw[opacity=0.2] (0,0) -- (0,1);
					\draw[fill=lightgray] (0.5,0.5){} +(0:\rt) circle(\r);
					\draw[fill=lightgray] (0.5,0.5){} +(180:\rt) circle(\r);
				\end{scope}
				\begin{scope}[shift = ($-1*(a)$)]
					\draw[opacity=0.2] (0,0) -- (1,0);
					\draw[opacity=0.2] (0,0) -- (0,1);
					\draw[fill=lightgray] (0.5,0.5){} +(0:\rt) circle(\r);
					\draw[fill=lightgray] (0.5,0.5){} +(180:\rt) circle(\r);
				\end{scope}
				\begin{scope}[shift = ($-1*(b)$)]
					\draw[opacity=0.2] (0,0) -- (1,0);
					\draw[opacity=0.2] (0,0) -- (0,1);
					\draw[fill=lightgray] (0.5,0.5){} +(0:\rt) circle(\r);
					\draw[fill=lightgray] (0.5,0.5){} +(180:\rt) circle(\r);
				\end{scope}
				\begin{scope}[shift = ($(a)+(b)$)]
					\draw[opacity=0.2] (1,0) -- (1,1) -- (0,1);
					\draw[opacity=0.2] (0,0) -- (1,0);
					\draw[opacity=0.2] (0,0) -- (0,1);
					\draw[fill=lightgray] (0.5,0.5){} +(0:\rt) circle(\r);
					\draw[fill=lightgray] (0.5,0.5){} +(180:\rt) circle(\r);
				\end{scope}
				\begin{scope}[shift = ($-1*(a)-(b)$)]
					\draw[opacity=0.2] (0,0) -- (1,0);
					\draw[opacity=0.2] (0,0) -- (0,1);
					\draw[fill=lightgray] (0.5,0.5){} +(0:\rt) circle(\r);
					\draw[fill=lightgray] (0.5,0.5){} +(180:\rt) circle(\r);
				\end{scope}
				\begin{scope}[shift = ($(a)-(b)$)]
					\draw[opacity=0.2] (1,0) -- (1,1);
					\draw[opacity=0.2] (0,0) -- (1,0);
					\draw[opacity=0.2] (0,0) -- (0,1);
					\draw[fill=lightgray] (0.5,0.5){} +(0:\rt) circle(\r);
					\draw[fill=lightgray] (0.5,0.5){} +(180:\rt) circle(\r);
				\end{scope}
				\begin{scope}[shift = ($-1*(a)+(b)$)]
					\draw[opacity=0.2] (1,1) -- (0,1);
					\draw[opacity=0.2] (0,0) -- (1,0);
					\draw[opacity=0.2] (0,0) -- (0,1);
					\draw[fill=lightgray] (0.5,0.5){} +(0:\rt) circle(\r);
					\draw[fill=lightgray] (0.5,0.5){} +(180:\rt) circle(\r);
				\end{scope}
				\begin{scope}[shift = ($2*(a)$)]
					\draw (0.5,0.5) node[rotate=0]{$\cdots$};
				\end{scope}
				\begin{scope}[shift = ($-2*(a)$)]
					\draw (0.5,0.5) node[rotate=0]{$\cdots$};
				\end{scope}
				\begin{scope}[shift = ($2*(b)$)]
					\draw (0.5,0.3) node[rotate=90]{$\cdots$};
				\end{scope}
				\begin{scope}[shift = ($-2*(b)$)]
					\draw (0.5,0.7) node[rotate=90]{$\cdots$};
				\end{scope}
			\end{scope}
		\end{tikzpicture}
		\caption{Infinite, periodic system with dimers in a square lattice.}
	\end{subfigure}
	\caption{Illustration of the square lattice of dimers, which may support an exceptional point.} \label{fig:square}
\end{figure}
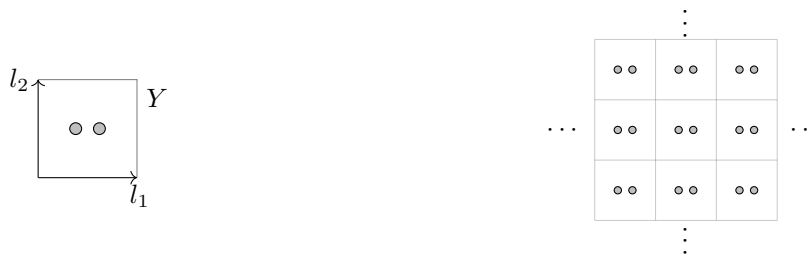

\begin{figure}[h] 
	\begin{subfigure}[b]{0.3\linewidth}
		\begin{center}
			\includegraphics[width=1\linewidth]{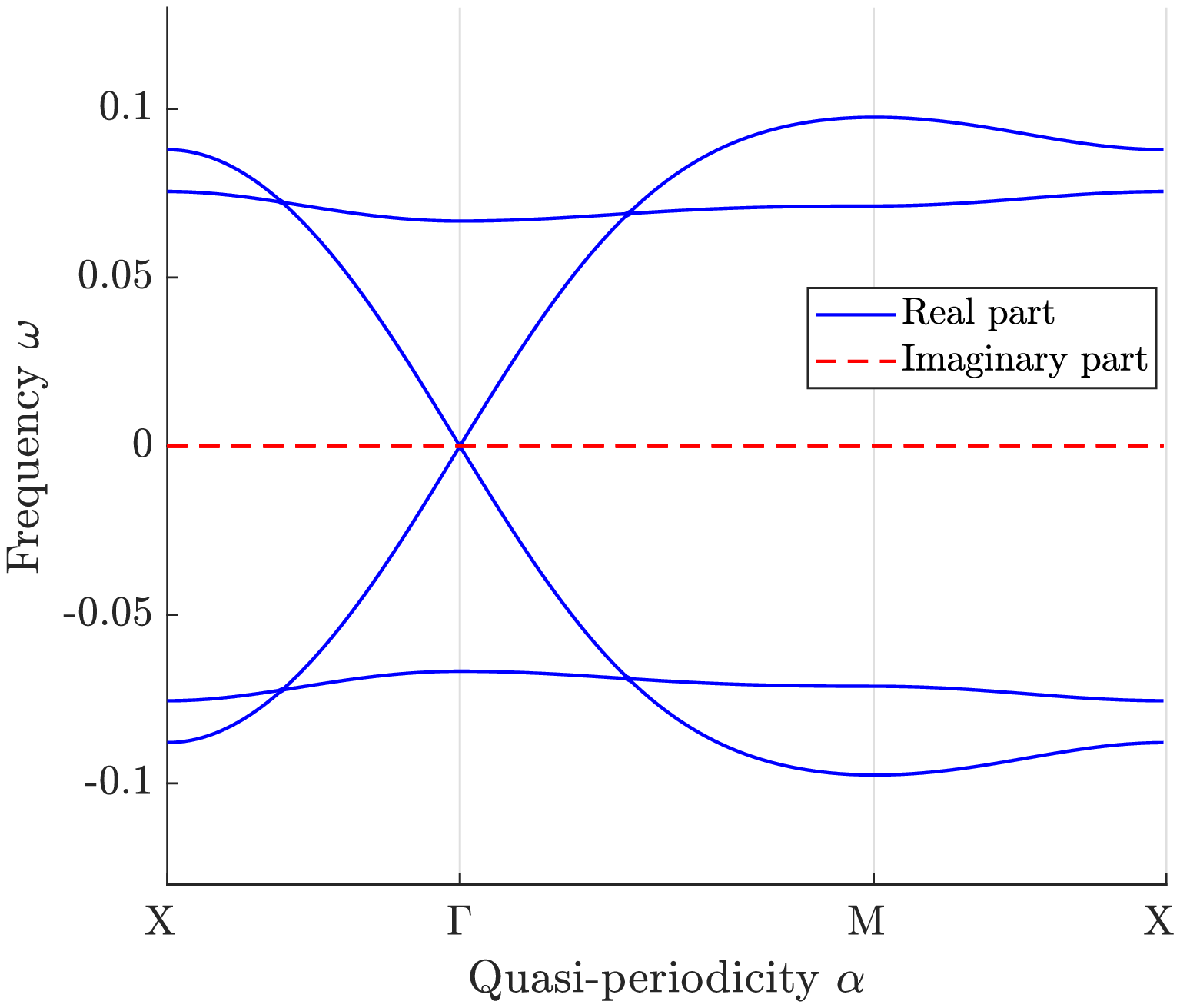}
		\end{center}
		\caption{Static band structure of the square lattice of dimers, folded with $\Omega = 0.26$.} \label{fig:SLresa}
	\end{subfigure}
	\hspace{10pt}
	\begin{subfigure}[b]{0.3\linewidth}
		\begin{center}
			\includegraphics[width=1\linewidth]{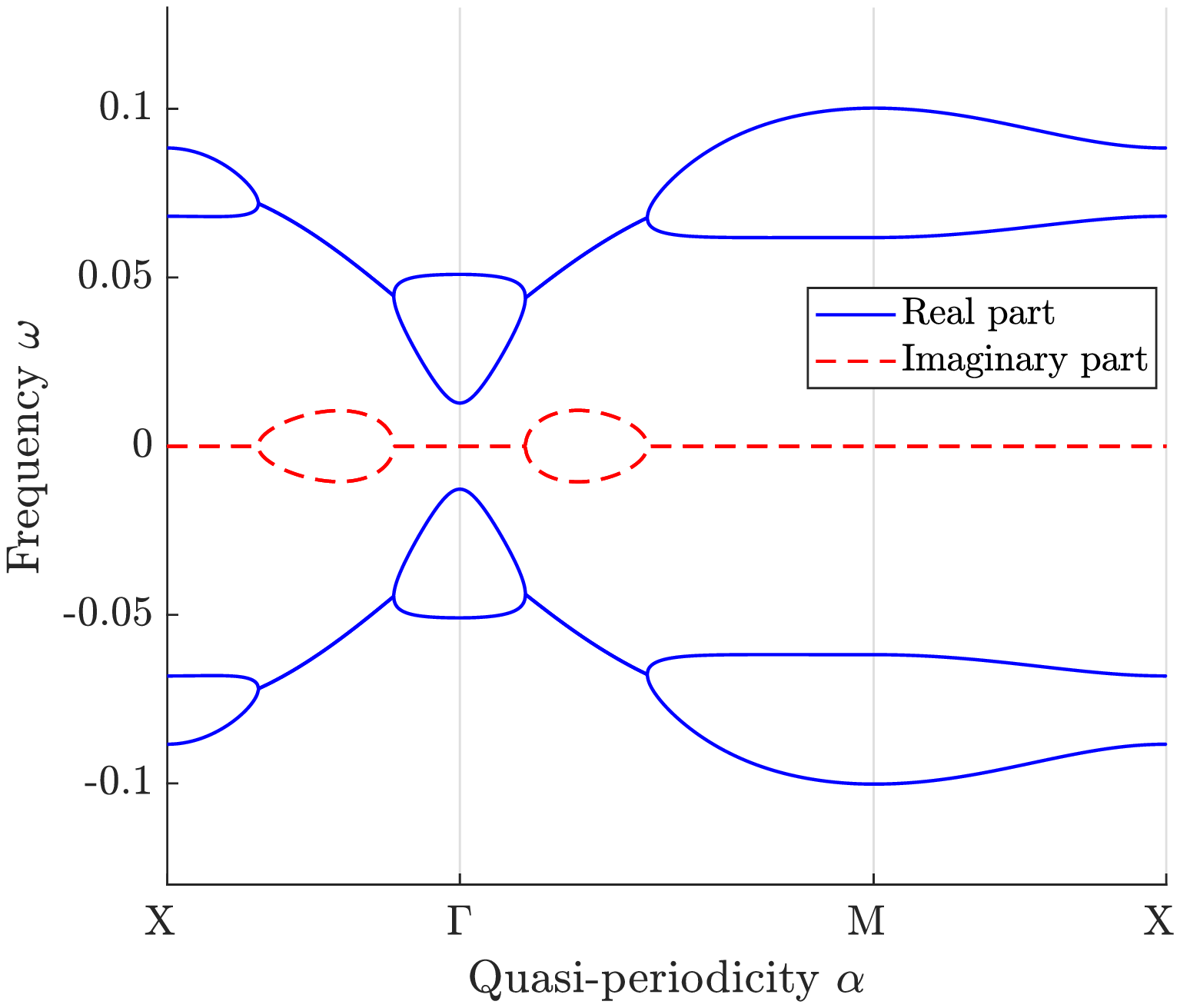}
		\end{center}
		\caption{$\varepsilon = 0.2$ and $\Omega = 0.26$, showing exceptional point degeneracies at some points in the Brillouin zone.}\label{fig:SLresb}
	\end{subfigure}
	\hspace{10pt}
	\begin{subfigure}[b]{0.3\linewidth}
		\begin{center}
			\includegraphics[width=1\linewidth]{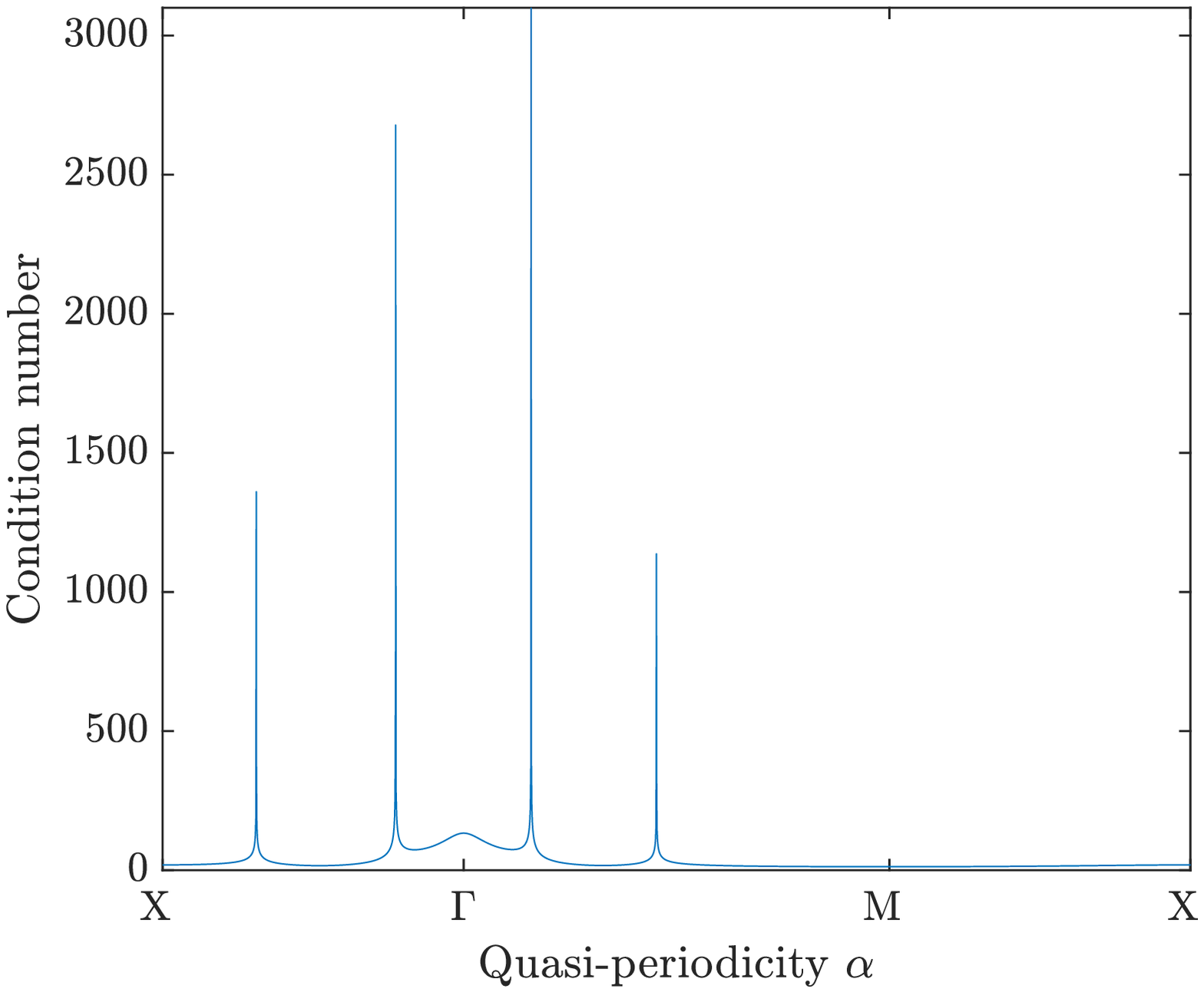}
		\end{center}
		\caption{Condition number of the eigenvector matrix of $\W$, showing a defective matrix at the degenerate points.}\label{fig:SLresc}
	\end{subfigure}
	\caption{Band structure of high-contrast dimers of resonators in a square lattice. For nonzero modulation strengths $\varepsilon$, the bands form exceptional point degeneracies where the system is defective and the spectrum changes from being real to being conjugate-symmetric.}\label{fig:SLres}
\end{figure}

\subsubsection{Dirac cone degeneracy at $\Gamma$ in trimer honeycomb lattice}
Next, we consider a honeycomb lattice of resonator trimers as illustrated in \Cref{fig:honeycomb}, similar to structures considered in \cite{koutserimpas2018zero,fleury2016floquet}. We use the same lattice as in \Cref{sec:HL_uni}, where the unit cell now contains six resonators $D_i$ respectively centred at $c_i$, $i=1,..,6$, given by
\begin{align*}
	c_1 &= (1,0) + 3R(1,0), \quad c_2 = (1,0) + 3R\left(\cos\left(\frac{2\pi}{3}\right),\sin\left(\frac{2\pi}{3}\right)\right),   &c_3 = (1,0) + 3R\left(\cos\left(\frac{4\pi}{3}\right),\sin\left(\frac{4\pi}{3}\right)\right),\\
	c_4 &= (2,0) + 3R\left(\cos\left(\frac{\pi}{3}\right),\sin\left(\frac{\pi}{3}\right)\right), \qquad c_5 = (2,0) - 3R(1,0), 	 &c_6 = (2,0) + 3R\left(\cos\left(\frac{5\pi}{3}\right),\sin\left(\frac{5\pi}{3}\right)\right).
\end{align*}
We use the modulation given by $\kappa_i(t) = 1, \ i=1,...,6$ and 
\begin{equation}\rho_1(t) = \rho_4(t) = \frac{1}{1 + \varepsilon\cos(\Omega t)}, \quad \rho_2(t) = \rho_5(t) = \frac{1}{1 + \varepsilon\cos\left(\Omega t + \frac{2\pi}{3}\right)}, \quad \rho_3(t) = \rho_6(t) = \frac{1}{1 + \varepsilon\cos\left(\Omega t + \frac{4\pi}{3}\right)},\end{equation}
for $0 \leq \varepsilon < 1$.

The band structure of the material is presented in \Cref{fig:HLres} with modulation frequency $\Omega = 0.15$. In the static case, the band structure is folded and, as expected, exhibits a Dirac cone at $\alpha = \mathrm{K}$ \cite{ammari2020honeycomb}. As $\varepsilon$ increases, the gap between the $4$\textsuperscript{th} and the $5$\textsuperscript{th} bands at $\Gamma$ decreases. At a specific point, namely $\varepsilon = 0.3$, the gap closes in a Dirac cone at $\Gamma$. We remark that this linear dispersion at the origin of the Brillouin zone is a prerequisite for creating double-zero index materials \cite{koutserimpas2018zero,ammari2020highfrequency}.
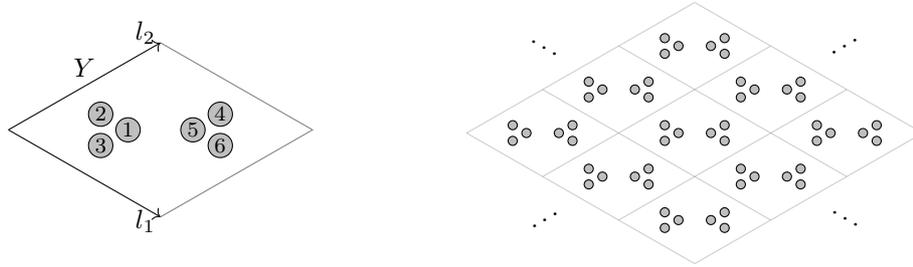
\begin{figure}[tbh]
	\begin{subfigure}[b]{0.48\linewidth}
		\centering
		\begin{tikzpicture}[scale=2]
				\pgfmathsetmacro{\r}{0.08pt}
				\pgfmathsetmacro{\rt}{0.12pt}
				\coordinate (a) at (1,{1/sqrt(3)});		
				\coordinate (b) at (1,{-1/sqrt(3)});	
				\coordinate (c) at (2,0);
				\coordinate (x1) at ({2/3},0);
				\coordinate (x2) at ({4/3},0);
				\draw[->] (0,0) -- (a) node[pos=0.9,xshift=0,yshift=7]{ $l_2$} node[pos=0.5,above]{$Y$};
				\draw[->] (0,0) -- (b) node[pos=0.9,xshift=0,yshift=-5]{ $l_1$};
				\draw[opacity=0.5] (a) -- (c) -- (b);
				\draw[fill=lightgray] (x1){} +(0:\rt) circle(\r) node{\footnotesize  1};
				\draw[fill=lightgray] (x1){} +(120:\rt) circle(\r)node{\footnotesize  2};
				\draw[fill=lightgray] (x1){} +(240:\rt) circle(\r) node{\footnotesize  3};
				\draw[fill=lightgray] (x2){} +(60:\rt) circle(\r) node{\footnotesize  4};
				\draw[fill=lightgray] (x2){} +(180:\rt) circle(\r) node{\footnotesize  5};
				\draw[fill=lightgray] (x2){} +(300:\rt) circle(\r) node{\footnotesize  6};
		\end{tikzpicture}
		\vspace{0.3cm}
		\caption{Hexagonal lattice unit cell $Y$ containing $6$ resonators.}
	\end{subfigure}
	\begin{subfigure}[b]{0.48\linewidth}
		\begin{tikzpicture}[scale=1]
			\begin{scope}[xshift=-5cm,scale=1]
				\pgfmathsetmacro{\r}{0.06pt}
				\pgfmathsetmacro{\rt}{0.12pt}
				\coordinate (a) at (1,{1/sqrt(3)});		
				\coordinate (b) at (1,{-1/sqrt(3)});
				
				\draw[opacity=0.2] (0,0) -- (a);
				\draw[opacity=0.2] (0,0) -- (b);
				\draw[fill=lightgray] ({2/3},0){} +(0:\rt) circle(\r);
				\draw[fill=lightgray] ({2/3},0){} +(120:\rt) circle(\r);
				\draw[fill=lightgray] ({2/3},0){} +(240:\rt) circle(\r);
				\draw[fill=lightgray] ({4/3},0){} +(60:\rt) circle(\r);
				\draw[fill=lightgray] ({4/3},0){} +(180:\rt) circle(\r);
				\draw[fill=lightgray] ({4/3},0){} +(300:\rt) circle(\r);
				
				\begin{scope}[shift = (a)]					
					\draw[opacity=0.2] (0,0) -- (1,{1/sqrt(3)});
					\draw[opacity=0.2] (0,0) -- (1,{-1/sqrt(3)});
					\draw[opacity=0.2] (1,{1/sqrt(3)}) -- (2,0);
					\draw[fill=lightgray] ({2/3},0){} +(0:\rt) circle(\r);
					\draw[fill=lightgray] ({2/3},0){} +(120:\rt) circle(\r);
					\draw[fill=lightgray] ({2/3},0){} +(240:\rt) circle(\r);
					\draw[fill=lightgray] ({4/3},0){} +(60:\rt) circle(\r);
					\draw[fill=lightgray] ({4/3},0){} +(180:\rt) circle(\r);
					\draw[fill=lightgray] ({4/3},0){} +(300:\rt) circle(\r);
				\end{scope}
				\begin{scope}[shift = (b)]
					
					\draw[opacity=0.2] (0,0) -- (1,{1/sqrt(3)});
					\draw[opacity=0.2] (0,0) -- (1,{-1/sqrt(3)});
					\draw[opacity=0.2] (2,0) -- (1,{-1/sqrt(3)});
					\draw[fill=lightgray] ({2/3},0){} +(0:\rt) circle(\r);
					\draw[fill=lightgray] ({2/3},0){} +(120:\rt) circle(\r);
					\draw[fill=lightgray] ({2/3},0){} +(240:\rt) circle(\r);
					\draw[fill=lightgray] ({4/3},0){} +(60:\rt) circle(\r);
					\draw[fill=lightgray] ({4/3},0){} +(180:\rt) circle(\r);
					\draw[fill=lightgray] ({4/3},0){} +(300:\rt) circle(\r);
				\end{scope}
				\begin{scope}[shift = ($-1*(a)$)]
					
					\draw[opacity=0.2] (0,0) -- (1,{1/sqrt(3)});
					\draw[opacity=0.2] (0,0) -- (1,{-1/sqrt(3)});
					\draw[fill=lightgray] ({2/3},0){} +(0:\rt) circle(\r);
					\draw[fill=lightgray] ({2/3},0){} +(120:\rt) circle(\r);
					\draw[fill=lightgray] ({2/3},0){} +(240:\rt) circle(\r);
					\draw[fill=lightgray] ({4/3},0){} +(60:\rt) circle(\r);
					\draw[fill=lightgray] ({4/3},0){} +(180:\rt) circle(\r);
					\draw[fill=lightgray] ({4/3},0){} +(300:\rt) circle(\r);
				\end{scope}
				\begin{scope}[shift = ($-1*(b)$)]
					
					\draw[opacity=0.2] (0,0) -- (1,{1/sqrt(3)});
					\draw[opacity=0.2] (0,0) -- (1,{-1/sqrt(3)});
					\draw[fill=lightgray] ({2/3},0){} +(0:\rt) circle(\r);
					\draw[fill=lightgray] ({2/3},0){} +(120:\rt) circle(\r);
					\draw[fill=lightgray] ({2/3},0){} +(240:\rt) circle(\r);
					\draw[fill=lightgray] ({4/3},0){} +(60:\rt) circle(\r);
					\draw[fill=lightgray] ({4/3},0){} +(180:\rt) circle(\r);
					\draw[fill=lightgray] ({4/3},0){} +(300:\rt) circle(\r);
				\end{scope}
				\begin{scope}[shift = ($(a)+(b)$)]
					
					\draw[opacity=0.2] (0,0) -- (1,{1/sqrt(3)});
					\draw[opacity=0.2] (0,0) -- (1,{-1/sqrt(3)});
					\draw[opacity=0.2] (1,{1/sqrt(3)}) -- (2,0) -- (1,{-1/sqrt(3)});
					\draw[fill=lightgray] ({2/3},0){} +(0:\rt) circle(\r);
					\draw[fill=lightgray] ({2/3},0){} +(120:\rt) circle(\r);
					\draw[fill=lightgray] ({2/3},0){} +(240:\rt) circle(\r);
					\draw[fill=lightgray] ({4/3},0){} +(60:\rt) circle(\r);
					\draw[fill=lightgray] ({4/3},0){} +(180:\rt) circle(\r);
					\draw[fill=lightgray] ({4/3},0){} +(300:\rt) circle(\r);
				\end{scope}
				\begin{scope}[shift = ($-1*(a)-(b)$)]
					
					\draw[opacity=0.2] (0,0) -- (1,{1/sqrt(3)});
					\draw[opacity=0.2] (0,0) -- (1,{-1/sqrt(3)});
					\draw[fill=lightgray] ({2/3},0){} +(0:\rt) circle(\r);
					\draw[fill=lightgray] ({2/3},0){} +(120:\rt) circle(\r);
					\draw[fill=lightgray] ({2/3},0){} +(240:\rt) circle(\r);
					\draw[fill=lightgray] ({4/3},0){} +(60:\rt) circle(\r);
					\draw[fill=lightgray] ({4/3},0){} +(180:\rt) circle(\r);
					\draw[fill=lightgray] ({4/3},0){} +(300:\rt) circle(\r);
				\end{scope}
				\begin{scope}[shift = ($(a)-(b)$)]
					
					\draw[opacity=0.2] (0,0) -- (1,{1/sqrt(3)});
					\draw[opacity=0.2] (0,0) -- (1,{-1/sqrt(3)});
					\draw[opacity=0.2] (1,{1/sqrt(3)}) -- (2,0);
					\draw[fill=lightgray] ({2/3},0){} +(0:\rt) circle(\r);
					\draw[fill=lightgray] ({2/3},0){} +(120:\rt) circle(\r);
					\draw[fill=lightgray] ({2/3},0){} +(240:\rt) circle(\r);
					\draw[fill=lightgray] ({4/3},0){} +(60:\rt) circle(\r);
					\draw[fill=lightgray] ({4/3},0){} +(180:\rt) circle(\r);
					\draw[fill=lightgray] ({4/3},0){} +(300:\rt) circle(\r);
				\end{scope}
				\begin{scope}[shift = ($-1*(a)+(b)$)]					
					\draw[opacity=0.2] (0,0) -- (1,{1/sqrt(3)});
					\draw[opacity=0.2] (0,0) -- (1,{-1/sqrt(3)});
					\draw[opacity=0.2] (2,0) -- (1,{-1/sqrt(3)});
					\draw[fill=lightgray] ({2/3},0){} +(0:\rt) circle(\r);
					\draw[fill=lightgray] ({2/3},0){} +(120:\rt) circle(\r);
					\draw[fill=lightgray] ({2/3},0){} +(240:\rt) circle(\r);
					\draw[fill=lightgray] ({4/3},0){} +(60:\rt) circle(\r);
					\draw[fill=lightgray] ({4/3},0){} +(180:\rt) circle(\r);
					\draw[fill=lightgray] ({4/3},0){} +(300:\rt) circle(\r);
				\end{scope}
				\begin{scope}[shift = ($2*(a)$)]
					\draw (1,0) node[rotate=30]{$\cdots$};
				\end{scope}
				\begin{scope}[shift = ($-2*(a)$)]
					\draw (1,0) node[rotate=210]{$\cdots$};
				\end{scope}
				\begin{scope}[shift = ($2*(b)$)]
					\draw (1,0) node[rotate=-30]{$\cdots$};
				\end{scope}
				\begin{scope}[shift = ($-2*(b)$)]
					\draw (1,0) node[rotate=150]{$\cdots$};
				\end{scope}
			\end{scope}
		\end{tikzpicture}
		\caption{Periodic system with trimers in a honeycomb lattice.}
	\end{subfigure}
	\caption{Illustration of the ``artificial spin'' honeycomb lattice, which can support a Dirac cone degeneracy at $\Gamma$.} \label{fig:honeycomb}
\end{figure}

\begin{figure}[h] 
	\begin{subfigure}[b]{0.3\linewidth}
		\begin{center}
			\includegraphics[width=1\linewidth]{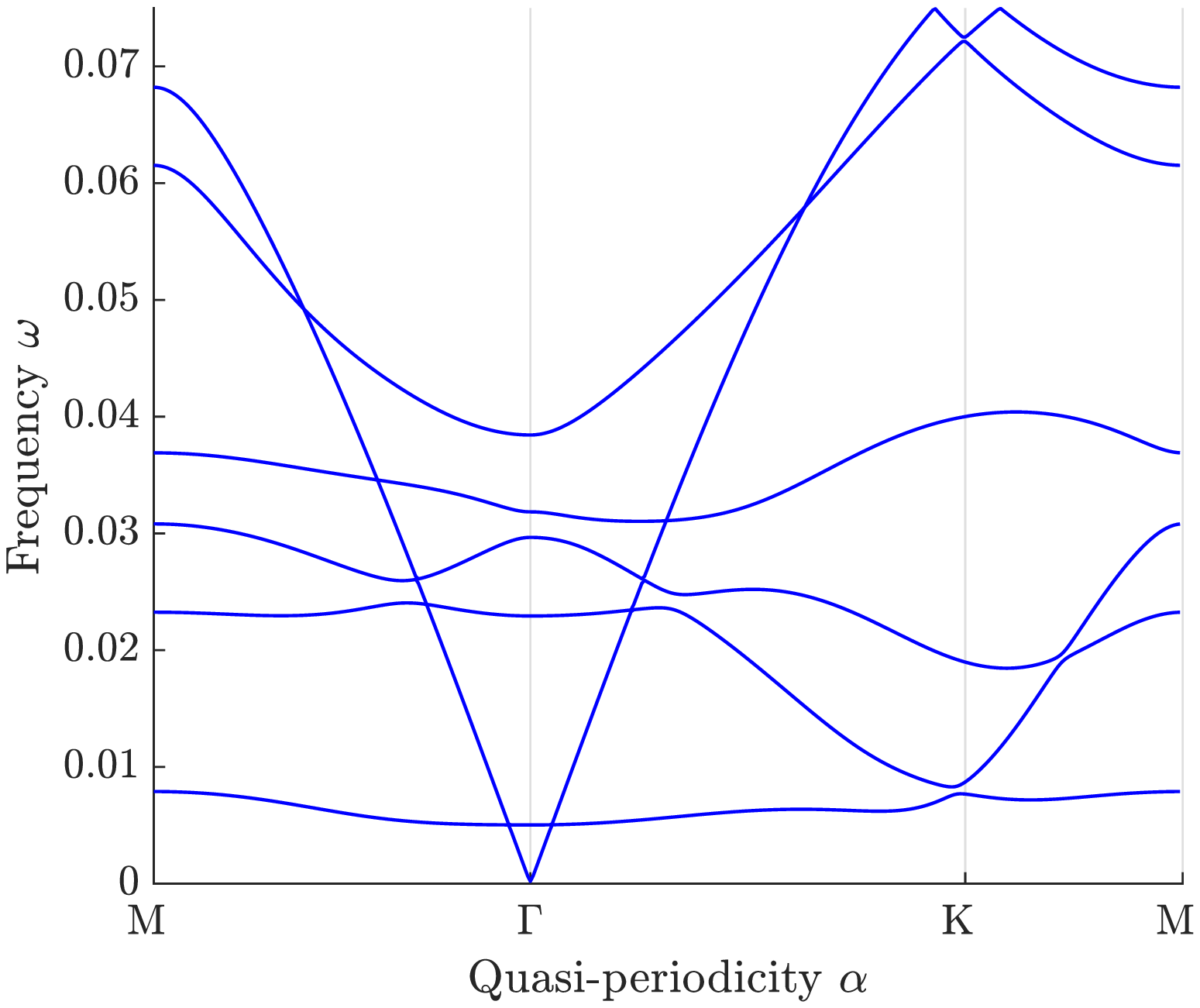}
		\end{center}
		\caption{Static, folded, band structure of the honeycomb lattice of trimers.\\} \label{fig:HLresa}
	\end{subfigure}
	\hspace{10pt}
	\begin{subfigure}[b]{0.3\linewidth}
		\begin{center}
			\includegraphics[width=1\linewidth]{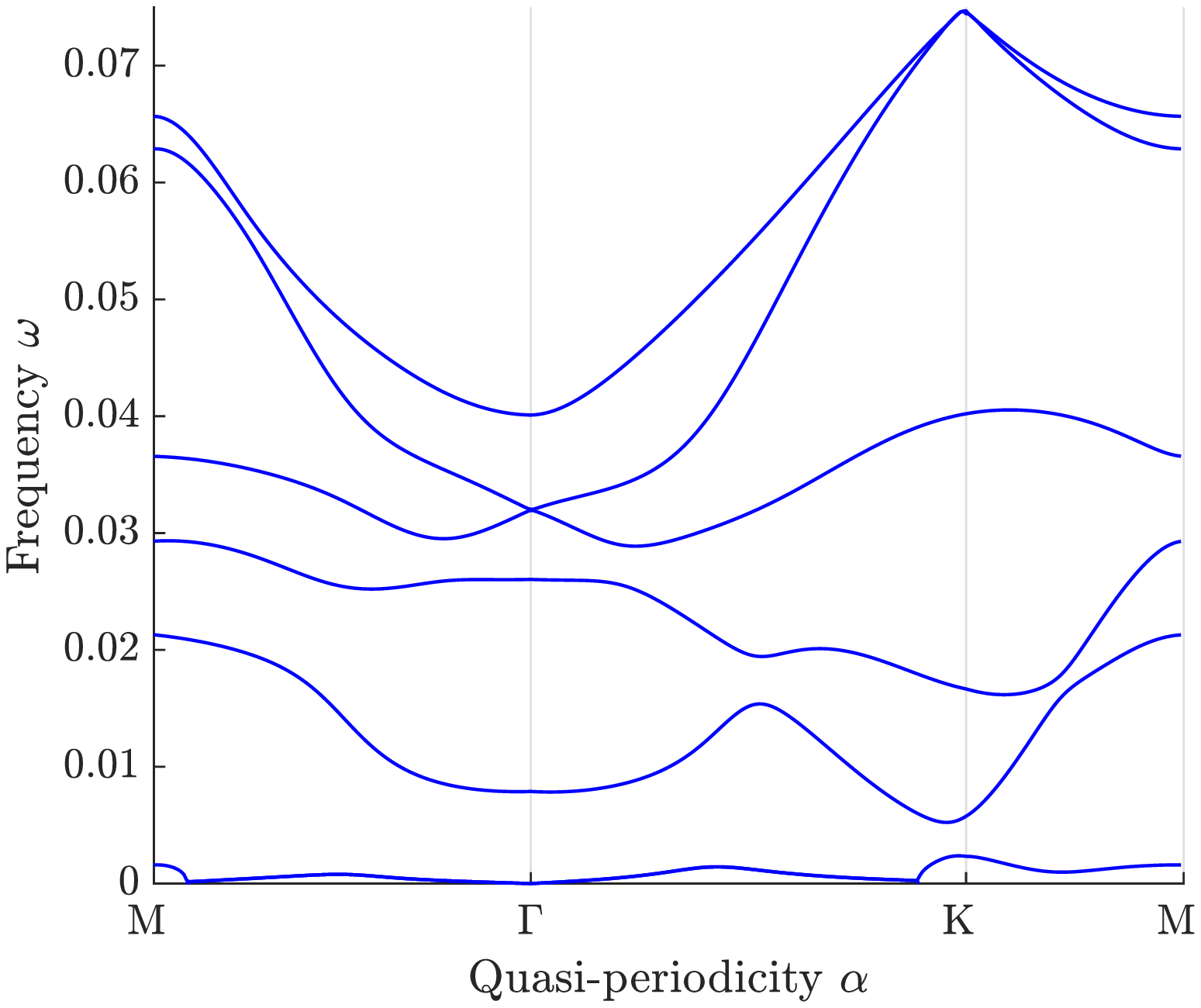}
		\end{center}
		\caption{Modulated band structure at $\varepsilon = 0.3$ and $\Omega = 0.15$, showing Dirac cones at both $\Gamma$ and $\mathrm{K}$.}\label{fig:HLresb}
	\end{subfigure}
	\hspace{10pt}
	\begin{subfigure}[b]{0.3\linewidth}
		\begin{center}
			\includegraphics[width=1\linewidth]{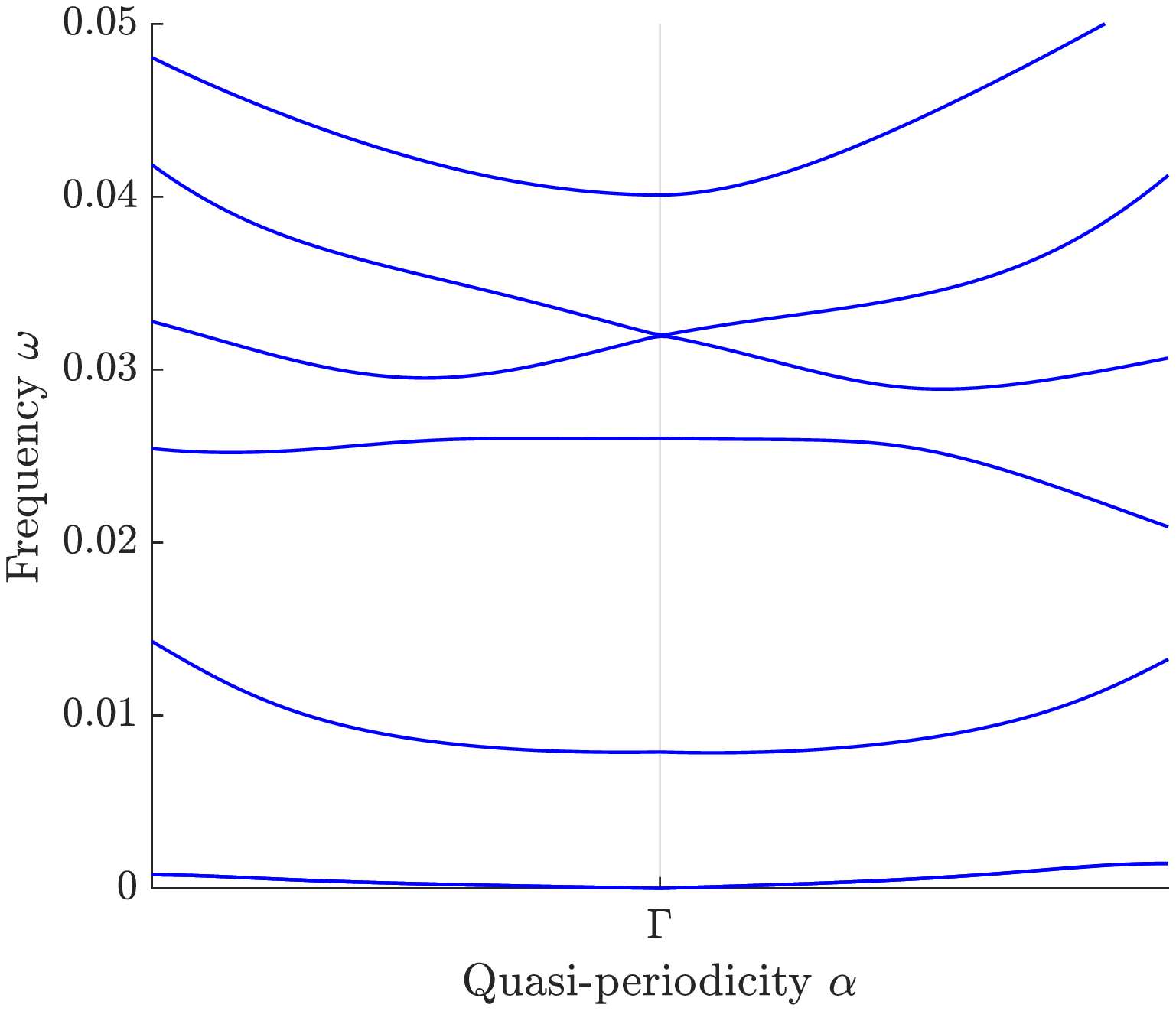}
		\end{center}
		\caption{Close-up of \Cref{fig:HLresb} around $\Gamma$, showing the Dirac cone of the $4$\textsuperscript{th} and $5$\textsuperscript{th} bands.} \label{fig:HLresc}
	\end{subfigure}
	\caption{Band structure of high-contrast resonators in a trimer honeycomb lattice. The static band structure (\Cref{fig:HLresa}) shows a Dirac cone at $\mathrm{K}$. As the modulation increases, the gap between the $4$\textsuperscript{th} and $5$\textsuperscript{th} bands closes, and at $\varepsilon = 0.3$ the gap closes in a Dirac cone. For clarity, only the positive real part of the band structure is shown.}\label{fig:HLres}
\end{figure}

\section{Concluding remarks}
In this work, we have provided a mathematical foundation for time-dependent systems of high-contrast subwavelength resonators. We have considered two types of time-modulation. In the case of uniform time-modulation, the wave equation is separable and the quasifrequency band structure is described as the quasifrequencies of a Hill differential equation in terms of the static band structure. In the case where the time-modulation only occurs inside the resonators, the subwavelength band structure admits a capacitance matrix characterization, which generalizes the recently derived characterization in the static case. We have exemplified both types of modulation numerically, and demonstrated $k$-gaps, exceptional points and Dirac cone degeneracies at the origin of the Brillouin zone, the latter which enables zero-refractive index materials in the subwavelength regime. 

\appendix
\section{Finite system of resonators} \label{sec:finite}
In this section, we derive the corresponding results in the setting of finite resonator systems, not repeated periodically. We let $D$ be a collection of $N$ disjoint domains, defined as in \Cref{sec:formulation}. Again, we compute the quasifrequencies $\omega$ of the wave equation \eqref{eq:wave} (we emphasize that there is no quasiperiodic condition in the spatial dimension in this case, and the quasifrequencies $\omega$ represent resonant frequencies and not band functions).

\subsection{Uniformly modulated systems}
As in \Cref{sec:uniform}, we assume that the material parameters are modulated by some envelopes $\kappa_t$, $\rho_t$ as follows:
\begin{equation}\label{eq:uniform_finite}
	\kappa(x,t) = \kappa_x(x)\kappa_t(t), \qquad \rho(x,t) = \rho_x(x)\rho_t(t),
\end{equation}
where $\rho_t$ is piecewise continuous and $\kappa_t\in C^1(\R)$. Now, we assume that the spatial parts $\kappa_x$ and $\rho_x$ satisfy 
\begin{equation}\label{eq:param_finite}
	\kappa_x(x) =  \begin{cases}
		\kappa_0, & x \in \R^d \setminus \overline{D}, \\ \kappa_\r, & x\in D,
	\end{cases} \qquad \rho_x(x) =  \begin{cases}
		\rho_0, & x \in \R^d \setminus \overline{D}, \\ \rho_\r, & x\in D.
	\end{cases}
\end{equation}
We seek solutions to \eqref{eq:wave} by separation of variables. Assume $u(x,t) = \Phi(t) v(x)$. Then we find from \eqref{eq:wave} that 
\begin{equation} \label{eq:time_finite}  \frac{\dx }{\dx t}\frac{1}{\kappa_t(t)} \frac{\dx}{\dx t}\Phi(t) +\frac{\omega^2}{\rho_t(t)}\Phi(t) = 0,
\end{equation}
and
\begin{equation} \label{eq:space_finite}
	\left\{
	\begin{array} {ll}
		\ds \Delta {v}+ \frac{\omega^2\rho_0}{\kappa_0} {v}  = 0 & \text{in } \R^d \setminus \overline{D}, \\[0.3em]
		\ds \Delta {v}+\frac{\omega^2\rho_\r}{\kappa_\r}{v}  = 0 & \text{in } D, \\
		\nm
		\ds  {v}|_{+} -{v}|_{-}  = 0  & \text{on } \partial D, \\
		\nm
		\ds  \delta \frac{\partial {v}}{\partial \nu} \bigg|_{+} - \frac{\partial {v}}{\partial \nu} \bigg|_{-} = 0 & \text{on } \partial D, \\
		\nm
	\end{array}
	\right.
\end{equation}
for some constant $\omega$. In the static case, \eqref{eq:space_finite} is usually coupled with the outgoing Sommerfeld radiation condition, in order to select the physical solution. In the modulated case however, due to lack of energy conservation, there may be both outgoing and incoming solutions. We denote the outgoing (respectively incoming) resonant frequencies by $\omega_{\mathrm{s},i}^+$ (respectively $\omega_{\mathrm{s},i}^-$) for $i=1,2,...$. We observe that $\omega_{\mathrm{s},i}^- = \overline{\omega_{\mathrm{s},i}^+}$. It is well-known that the first $2N$ resonant frequencies scale as $O(\delta^{1/2})$ and $N$ of these have positive real part (known as the \emph{Minnaert frequencies}) \cite{davies2019fully}. These are real to leading order, and we have $\omega_{\mathrm{s},i}^+ = \omega_{\mathrm{s},i}^- + O(\delta)$ for $i=1,2,...,2N$.  Substituting $\Phi(t) = \sqrt{\kappa_t(t)} \Psi(t)$ we can obtain the following result.
\begin{prop}
	Assume that the material parameters are given by \eqref{eq:uniform_finite} and \eqref{eq:param_finite}. Then, the quasifrequencies $\omega \in Y^*_t$ to the wave equation \eqref{eq:wave} in the subwavelength regime are given by the quasifrequencies of the Hill equation
	\begin{equation}\Psi''(t) +  \left(\left(\omega_i^\pm(t)\right)^2 + \frac{\sqrt{\kappa_t}}{2}\frac{\dx}{\dx t} \frac{\kappa_t'}{\kappa_t^{3/2}}\right)\Psi(t) = 0,\end{equation}
	for $i=1,2,...,$. Here, $\omega_i^\pm(t)$ are the instantaneous resonant frequencies defined by $\omega_i^\pm(t) = \omega_{\mathrm{s},i}^\pm\sqrt{\frac{\kappa_t(t)}{\rho_t(t)}}$.
\end{prop}

\subsection{Resonator-modulated systems}
We now consider the finite analogue of the system studied in \Cref{sec:resonatormod}. For simplicity, we will restrict to the case $d=3$ (the case $d=2$ requires a slightly different layer-potential analysis, as outlined in \cite[Appendix B]{ammari2018minnaert}). We assume
\begin{equation} \label{eq:resonatormod_finite}
	\kappa(x,t) = \begin{cases}
		\kappa_0, & x \in \R^d \setminus \overline{D}, \\  \kappa_\r\kappa_i(t), & x\in D_i,
	\end{cases}, \qquad \rho(x,t) = \begin{cases}
		\rho_0, & x \in \R^d \setminus \overline{D}, \\  \rho_\r\rho_i(t), & x\in D_i. \end{cases}
\end{equation}
We denote the outgoing (respectively incoming) Helmholtz Green's functions by $G^{k,+}$ (respectively $G^{k,-}$), defined by
\begin{equation}
G^{k,\pm}(x,y) := -\frac{e^{\pm\iu k|x-y|}}{4\pi|x-y|}, \quad x,y \in \R^3, x\neq y, k\in \mathbb{C}.
\end{equation}
Let $D\in \R^3$ be as in \Cref{sec:formulation}. We introduce the single layer potential $\mathcal{S}_{D}^{k,\pm}: L^2(\partial D) \rightarrow H_{\textrm{loc}}^1(\R^3)$, defined by
\begin{equation*}
	\mathcal{S}_D^{k,\pm}[\phi](x) := \int_{\partial D} G^{k,\pm}(x,y)\phi(y) \dx \sigma(y), \quad x \in \R^3.
\end{equation*}
We observe that $\mathcal{S}_D^{0,+} = \mathcal{S}_D^{0,-} =: \mathcal{S}_D$. Taking the trace on $\p D$, it is well-known that in dimension three  $\mathcal{S}_D: L^2(\p D) \rightarrow H^1(\p D)$ is invertible. We define the basis functions $\psi_i$ and the capacitance coefficients $C_{ij}$ as
\begin{equation}\psi_i = \left(\mathcal{S}_D\right)^{-1}[\chi_{\p D_i}], \qquad C_{ij}= -\int_{\p D_i} \psi_j   \dx \sigma,\end{equation}
for $i,j=1,...,N$. The capacitance matrix $C$ is defined as the matrix $C = \left(C_{ij}\right)$. Following exactly the same steps as those in \Cref{sec:resonatormod}, we can then prove the following result.
\begin{thm}
	Assume that the material parameters are given by \eqref{eq:resonatormod_finite}. Then, as $\delta \to 0$, the quasifrequencies $\omega \in Y^*_t$ to the wave equation \eqref{eq:wave} in the subwavelength regime are, to leading order, given by the quasifrequencies of the system of ordinary differential equations
	\begin{equation}\label{eq:C_ODE_finite}
		\sum_{j=1}^N C_{ij} c_j(t) = \frac{|D_i|\rho_\r}{\delta\kappa_\r}\frac{1}{\rho_i(t)}\frac{\dx}{\dx t}\left(\frac{1}{\kappa_i(t)}\frac{\dx \rho_ic_i}{\dx t}\right),
	\end{equation}
	for $i=1,...,N$.
\end{thm}

\bibliographystyle{abbrv}
\bibliography{time}{}

\begin{thebibliography}{10}

\bibitem{davies2019fully}
H.~Ammari and B.~Davies.
\newblock A fully coupled subwavelength resonance approach to filtering
  auditory signals.
\newblock {\em Proc. R. Soc. A}, 475(2228):20190049, 2019.

\bibitem{ammari2020exceptional}
H.~Ammari, B.~Davies, E.~O. Hiltunen, H.~Lee, and S.~Yu.
\newblock Exceptional points in parity--time-symmetric subwavelength
  metamaterials.
\newblock {\em arXiv preprint arXiv:2003.07796}, 2020.

\bibitem{ammari2020highorder}
H.~Ammari, B.~Davies, E.~O. Hiltunen, H.~Lee, and S.~Yu.
\newblock High-order exceptional points and enhanced sensing in subwavelength
  resonator arrays.
\newblock {\em Stud. Appl. Math.}, 146:440--462, 2021.

\bibitem{ammari2018minnaert}
H.~Ammari, B.~Fitzpatrick, D.~Gontier, H.~Lee, and H.~Zhang.
\newblock Minnaert resonances for acoustic waves in bubbly media.
\newblock {\em Annales de l'Institut Henri Poincar{\'e} C, Analyse non
  lin{\'e}aire}, 35(7):1975--1998, 2018.

\bibitem{ammari2020honeycomb}
H.~Ammari, B.~Fitzpatrick, E.~O. Hiltunen, H.~Lee, and S.~Yu.
\newblock Honeycomb-lattice minnaert bubbles.
\newblock {\em SIAM Journal on Mathematical Analysis}, 52(6):5441--5466, 2020.

\bibitem{MaCMiPaP}
H.~Ammari, B.~Fitzpatrick, H.~Kang, M.~Ruiz, S.~Yu, and H.~Zhang.
\newblock {\em Mathematical and Computational Methods in Photonics and
  Phononics}, volume 235 of {\em Mathematical Surveys and Monographs}.
\newblock American Mathematical Society, Providence, 2018.

\bibitem{ammari2017subwavelength}
H.~Ammari, B.~Fitzpatrick, H.~Lee, S.~Yu, and H.~Zhang.
\newblock Subwavelength phononic bandgap opening in bubbly media.
\newblock {\em Journal of Differential Equations}, 263(9):5610--5629, 2017.

\bibitem{ammari2017double}
H.~Ammari, B.~Fitzpatrick, H.~Lee, S.~Yu, and H.~Zhang.
\newblock Double-negative acoustic metamaterials.
\newblock {\em Quart. Appl. Math.}, 77(4):767--791, 2019.

\bibitem{ammari2020highfrequency}
H.~Ammari, E.~O. Hiltunen, and S.~Yu.
\newblock A high-frequency homogenization approach near the dirac points in
  bubbly honeycomb crystals.
\newblock {\em Archive for Rational Mechanics and Analysis}, 238(3):1559--1583,
  2020.

\bibitem{ammari2009layer}
H.~Ammari, H.~Kang, and H.~Lee.
\newblock {\em Layer potential techniques in spectral analysis}, volume 153 of
  {\em Mathematical Surveys and Monographs}.
\newblock American Mathematical Society, Providence, 2009.

\bibitem{bal2019time}
G.~Bal, M.~Fink, and O.~Pinaud.
\newblock Time-reversal by time-dependent perturbations.
\newblock {\em SIAM Journal on Applied Mathematics}, 79(3):754--780, 2019.

\bibitem{carlson2000eigenvalue}
R.~Carlson.
\newblock Eigenvalue estimates and trace formulas for the matrix hill's
  equation.
\newblock {\em Journal of Differential Equations}, 167(1):211--244, 2000.

\bibitem{cassedy1967dispersion}
E.~S. Cassedy.
\newblock Dispersion relations in time-space periodic media part ii—unstable
  interactions.
\newblock {\em Proceedings of the IEEE}, 55(7):1154--1168, 1967.

\bibitem{cullen1958travelling}
A.~Cullen.
\newblock A travelling-wave parametric amplifier.
\newblock {\em Nature}, 181(4605):332--332, 1958.

\bibitem{denk1995floquet}
R.~Denk.
\newblock On the floquet exponents of hill's equation systems.
\newblock {\em Mathematische Nachrichten}, 172(1):87--94, 1995.

\bibitem{fleury2016floquet}
R.~Fleury, A.~B. Khanikaev, and A.~Al{\`u}.
\newblock Floquet topological insulators for sound.
\newblock {\em Nature communications}, 7(1):1--11, 2016.

\bibitem{guo2019nonreciprocal}
X.~Guo, Y.~Ding, Y.~Duan, and X.~Ni.
\newblock Nonreciprocal metasurface with space--time phase modulation.
\newblock {\em Light: Science \& Applications}, 8(1):1--9, 2019.

\bibitem{heiss2012physics}
W.~Heiss.
\newblock The physics of exceptional points.
\newblock {\em J. Phys. A: Math. Theor.}, 45(44):444016, 2012.

\bibitem{koutserimpas2018electromagnetic}
T.~T. Koutserimpas and R.~Fleury.
\newblock Electromagnetic waves in a time periodic medium with step-varying
  refractive index.
\newblock {\em IEEE Transactions on Antennas and Propagation},
  66(10):5300--5307, 2018.

\bibitem{koutserimpas2018zero}
T.~T. Koutserimpas and R.~Fleury.
\newblock Zero refractive index in time-floquet acoustic metamaterials.
\newblock {\em Journal of Applied Physics}, 123(9):091709, 2018.

\bibitem{koutserimpas2020electromagnetic}
T.~T. Koutserimpas and R.~Fleury.
\newblock Electromagnetic fields in a time-varying medium: Exceptional points
  and operator symmetries.
\newblock {\em IEEE Transactions on Antennas and Propagation}, 2020.

\bibitem{kuchment}
P.~Kuchment.
\newblock {\em Floquet Theory for Partial Differential Equations}.
\newblock Number~60 in Operator Theory: Advances and Applications. Birkh\"auser
  Verlag, Basel, 1993.

\bibitem{li2019nonreciprocal}
J.~Li, C.~Shen, X.~Zhu, Y.~Xie, and S.~A. Cummer.
\newblock Nonreciprocal sound propagation in space-time modulated media.
\newblock {\em Physical Review B}, 99(14):144311, 2019.

\bibitem{lions1988controlabilite}
J.-L. Lions.
\newblock Contr{\^o}labilit{\'e} exacte, perturbations et stabilisation de
  syst{\`e}mes distribu{\'e}s. tome 1.
\newblock {\em RMA}, 8, 1988.

\bibitem{ma2016acoustic}
G.~Ma and P.~Sheng.
\newblock Acoustic metamaterials: From local resonances to broad horizons.
\newblock {\em Sci. Adv.}, 2(2):e1501595, 2016.

\bibitem{martinez2016temporal}
J.~S. Mart{\'\i}nez-Romero, O.~Becerra-Fuentes, and P.~Halevi.
\newblock Temporal photonic crystals with modulations of both permittivity and
  permeability.
\newblock {\em Physical Review A}, 93(6):063813, 2016.

\bibitem{martinez2017standing}
J.~S. Mart{\'\i}nez-Romero and P.~Halevi.
\newblock Standing waves with infinite group velocity in a temporally periodic
  medium.
\newblock {\em Physical Review A}, 96(6):063831, 2017.

\bibitem{mclachlan1951theory}
N.~W. McLachlan.
\newblock Theory and application of mathieu functions.
\newblock {\em Publisher to the University Geoffrey Cumberlege, Oxford
  University Press}, 1951.

\bibitem{mendoncca2002time}
J.~Mendon{\c{c}}a and P.~Shukla.
\newblock Time refraction and time reflection: two basic concepts.
\newblock {\em Physica Scripta}, 65(2):160, 2002.

\bibitem{morgenthaler1958velocity}
F.~R. Morgenthaler.
\newblock Velocity modulation of electromagnetic waves.
\newblock {\em IRE Transactions on microwave theory and techniques},
  6(2):167--172, 1958.

\bibitem{nash2015topological}
L.~M. Nash, D.~Kleckner, A.~Read, V.~Vitelli, A.~M. Turner, and W.~T. Irvine.
\newblock Topological mechanics of gyroscopic metamaterials.
\newblock {\em Proceedings of the National Academy of Sciences},
  112(47):14495--14500, 2015.

\bibitem{nassar2017non}
H.~Nassar, H.~Chen, A.~Norris, M.~Haberman, and G.~Huang.
\newblock Non-reciprocal wave propagation in modulated elastic metamaterials.
\newblock {\em Proceedings of the Royal Society A: Mathematical, Physical and
  Engineering Sciences}, 473(2202):20170188, 2017.

\bibitem{nassar2018quantization}
H.~Nassar, H.~Chen, A.~Norris, and G.~Huang.
\newblock Quantization of band tilting in modulated phononic crystals.
\newblock {\em Physical Review B}, 97(1):014305, 2018.

\bibitem{ourir2019active}
A.~Ourir and M.~Fink.
\newblock Active control of the spoof plasmon propagation in time varying and
  non-reciprocal metamaterial.
\newblock {\em Scientific reports}, 9(1):1--8, 2019.

\bibitem{psiachos2021band}
D.~Psiachos and M.~Sigalas.
\newblock Band-gap tuning in two-dimensional spatiotemporal phononic crystals.
\newblock {\em Physical Review Applied}, 15(1):014022, 2021.

\bibitem{raghu2008analogs}
S.~Raghu and F.~D.~M. Haldane.
\newblock Analogs of quantum-hall-effect edge states in photonic crystals.
\newblock {\em Physical Review A}, 78(3):033834, 2008.

\bibitem{rechtsman2013photonic}
M.~C. Rechtsman, J.~M. Zeuner, Y.~Plotnik, Y.~Lumer, D.~Podolsky, F.~Dreisow,
  S.~Nolte, M.~Segev, and A.~Szameit.
\newblock Photonic floquet topological insulators.
\newblock {\em Nature}, 496(7444):196--200, 2013.

\bibitem{sounas2017non}
D.~L. Sounas and A.~Al{\`u}.
\newblock Non-reciprocal photonics based on time modulation.
\newblock {\em Nature Photonics}, 11(12):774--783, 2017.

\bibitem{strang2005characteristic}
J.-E. Str{\"a}ng.
\newblock On the characteristic exponents of {F}loquet solutions to the
  {M}athieu equation.
\newblock {\em Bulletins de l'Acad{\'e}mie Royale de Belgique}, 16(7):269--287,
  2005.

\bibitem{sugino2020nonreciprocal}
C.~Sugino, M.~Ruzzene, and A.~Erturk.
\newblock Nonreciprocal piezoelectric metamaterial framework and circuit
  strategies.
\newblock {\em Physical Review B}, 102(1):014304, 2020.

\bibitem{wang2019subwavelength}
L.~Wang, R.-Y. Zhang, B.~Hou, Y.~Huang, S.~Li, and W.~Wen.
\newblock Subwavelength topological edge states based on localized spoof
  surface plasmonic metaparticle arrays.
\newblock {\em Opt. Express}, 27(10):14407--14422, May 2019.

\bibitem{wilson2019temporally}
J.~Wilson, F.~Santosa, and P.~Martin.
\newblock Temporally manipulated plasmons on graphene.
\newblock {\em SIAM Journal on Applied Mathematics}, 79(3):1051--1074, 2019.

\bibitem{wilson2018temporal}
J.~Wilson, F.~Santosa, M.~Min, and T.~Low.
\newblock Temporal control of graphene plasmons.
\newblock {\em Physical Review B}, 98(8):081411, 2018.

\bibitem{xu2020physical}
X.~Xu, Q.~Wu, H.~Chen, H.~Nassar, Y.~Chen, A.~Norris, M.~R. Haberman, and
  G.~Huang.
\newblock Physical observation of a robust acoustic pumping in waveguides with
  dynamic boundary.
\newblock {\em Physical Review Letters}, 125(25):253901, 2020.

\bibitem{yves2017crystalline}
S.~Yves, R.~Fleury, T.~Berthelot, M.~Fink, F.~Lemoult, and G.~Lerosey.
\newblock Crystalline metamaterials for topological properties at subwavelength
  scales.
\newblock {\em Nat. Commun.}, 8:16023 EP --, Jul 2017.
\newblock Article.

\bibitem{yves2017topological}
S.~Yves, R.~Fleury, F.~Lemoult, M.~Fink, and G.~Lerosey.
\newblock Topological acoustic polaritons: robust sound manipulation at the
  subwavelength scale.
\newblock {\em New J. Phys.}, 19(7):075003, 2017.

\bibitem{zurita2009reflection}
J.~R. Zurita-S{\'a}nchez, P.~Halevi, and J.~C. Cervantes-Gonzalez.
\newblock Reflection and transmission of a wave incident on a slab with a
  time-periodic dielectric function $\epsilon$(t).
\newblock {\em Physical Review A}, 79(5):053821, 2009.

\end{thebibliography}
\end{document}